\documentclass{amsproc}
\usepackage[english,russian]{babel}

\usepackage{a4wide}
\usepackage{amsmath}
\allowdisplaybreaks
\usepackage{enumerate}

\usepackage[dvipsnames]{xcolor}

\usepackage{amsmath,amsthm}
\usepackage{amsfonts}
\usepackage{amssymb}
\usepackage{longtable}
\usepackage[matrix,arrow,curve]{xy}

\usepackage{amsmath}
\usepackage{amsfonts}
\usepackage{amssymb}

\newtheorem{theorem}{Theorem}
\newtheorem{cor}{Corollary}
\newtheorem{prop}{Proposition}

\newtheorem{lem}{Lemma}

{\bf}{\rm}

{\bf}{\rm}{\rm}

\newtheorem*{theorem*}{Theorem}
\newtheorem*{cor*}{Corollary}

\def\Real{\mathbb{R}}
\def\SO{\text{\rm SO}}

\def\PSL{\mathrm{PSL}_2(\mathbb{R})}
\def\PGL{\mathrm{PGL}_2(\mathbb{R})}
\def\sl{\mathfrak{sl}_2(\mathbb{R})}
\def\SAff{\mathrm{Aff}^0_1(\mathbb{R})}
\def\Aff{\mathrm{Aff}_1(\mathbb{R})}
\def\aff{\mathfrak{aff}_1(\mathbb{R})}

\newcommand{\Ll}{\mathcal{L}}

\newcommand{\fraka}{\mathfrak{a}}
\newcommand{\frakb}{\mathfrak{b}}

\newcommand{\frakg}{\mathfrak{g}}

\newcommand{\mathR}{\mathbb{R}}
\newcommand{\frakso}{\mathfrak{so}}
\newcommand{\frakco}{\mathfrak{co}}

\newcommand{\g}{\mathfrak{g}}
\newcommand{\h}{\mathfrak{h}}
\newcommand{\so}{\mathfrak{so}}
\newcommand{\co}{\mathfrak{co}}

\newcommand{\zr}{\ltimes}

\def\id{\mathop\text{\rm id}\nolimits}

\def\Hom{\mathop\text{\rm Hom}\nolimits}

\usepackage{stackrel}

\newcommand{\be}{\begin{equation}}
\newcommand{\ee}{\end{equation}}

\let\leq=\leqslant
\let\geq=\geqslant

\begin{document}
\selectlanguage{english}	

\title{Recurrent Lorentzian Weyl spaces}

\author{Andrei Dikarev}\thanks{$^1$Department of Mathematics and Statistics, Masaryk University, Faculty of Science, Kotl\'a\v{r}sk\'a 2, 611 37 Brno, Czech Republic}

\author{Anton S. Galaev}\thanks{$^2$University of Hradec Kr\'alov\'e, Department of Mathematics, Rokitansk\'eho 62, 500~03 Hradec Kr\'alov\'e,  Czech
Republic\\
E-mail: anton.galaev(at)uhk.cz}

\author{Eivind Schneider}\thanks{$^3$Department of Mathematics and Statistics, UiT The Arctic University of Norway, 9037 Troms\o, Norway.}

\begin{abstract}
We find the local form of all non-closed Lorentzian Weyl manifolds $(M,c,\nabla)$ with recurrent curvature tensor. The recurrent curvature tensor turns out to be  weighted parallel, i.e., the obtained spaces provide certain generalization of locally symmetric affine spaces for the Weyl geometry.  If the dimension of the manifold is greater than 3, then the conformal structure is flat, and the recurrent Weyl structure is locally determined by a single function of one variable. Two local structures are equivalent if and only if the corresponding functions are related by a transformation from $\SAff \times \PSL \times \mathbb Z_2$. We find generators for the field of rational scalar differential invariants of this Lie group action. The global structure of the manifold $M$ may be described in terms of a foliation with a transversal projective structure. It is shown that all locally homogeneous structures are locally equivalent, and there is only one simply connected homogeneous non-closed recurrent Lorentzian Weyl manifold. Moreover, there are 5 classes of cohomogeneity-one spaces, and all other spaces are of cohomogeneity-two. If $\dim M=3$, the non-closed recurrent Lorentzian Weyl structures are locally determined by one function of two variables or two functions of one variables, depending on whether its holonomy algebra is 1- or 2-dimensional. In this case, two structures with the same holonomy algebra are locally equivalent if and only if they are related, respectively, by a transformation from an infinite-dimensional Lie pseudogroup or a 4-dimensional subgroup of $\mathrm{Aff}(\mathbb R^3)$. Again we provide generators for the field of rational differential invariants. We find a local expression for the locally homogeneous non-closed recurrent Lorentzian Weyl manifolds of dimension 3, and also of those of cohomogeneity one and two. In the end we give a local description of the non-closed recurrent Lorentzian Weyl manifolds that are also Einstein-Weyl. All of them are 3-dimensional and have a 2-dimensional holonomy algebra. 
 
\vskip0.5cm


\vskip0.5cm



\end{abstract}

\maketitle
\tableofcontents

\section{Introduction} 
A Lorentzian Weyl manifold is a triple $(M,c,\nabla)$ where $M$ is a smooth manifold, $c$ is a conformal class of Lorentzian metrics on $M$, and $\nabla$ is a torsion-free affine connection on $M$ such that, for each $g\in c$, there exists a 1-form $\omega_g$ satisfying 
\begin{equation}\label{nablag}\nabla g=-2\omega_g\otimes g.\end{equation}

In what follows all manifolds will be assumed to be connected. 
Lorentzian Weyl manifolds have received a lot of attention in several recent works,  e.g., \cite{BFKN,DP,DMT,G-RBook,MOP}.
The connection $\nabla$ is uniquely determined by a metric $g\in c$ and the corresponding 1-form $\omega_g$ by the formula \begin{equation}\label{formulaK}\nabla=\nabla^g+K^g,\quad 	g(K^g_X(Y),Z)=g(Y,Z)\omega_g(X)+g(X,Z)\omega_g(Y)-g(X,Y)\omega_g(Z),\end{equation} where $\nabla^g$ is the Levi-Civita connection of the metric $g$, and $X,Y,Z$ are vector fields on $M$.
 If $h=e^{2f}g$ is a different metric in the conformal class, then the corresponding 1-form $\omega_h$ is given by
 $$\omega_h=\omega_g-df.$$ 
 The Weyl structure $(c,\nabla)$ is called closed if $d\omega_g=0$ for some (and thus every) $g\in c$. If the structure is closed, then $\nabla$ is locally the Levi-Civita connection for some metric in the conformal class.
  Let $R$ denote the curvature tensor of the affine connection $\nabla$.
Recall that an affine connection $\nabla$ is called locally symmetric if $\nabla R=0$. It turns out that

\begin{prop}
	If the connection $\nabla$ of a Lorentzian Weyl structure is locally symmetric, then the Weyl structure is closed.  
\end{prop}

Since a locally symmetric affine connection is (locally) the Levi-Civita connection for some metric in the conformal class, the study of locally symmetric Lorentzian Weyl manifolds reduces to the study of locally symmetric Lorentzian manifolds. In this paper we therefore relax this condition and consider Lorentzian Weyl structures with recurrent curvature tensors, i.e., curvature tensors satisfying
\begin{equation}
	\label{nablaR}
	\nabla R=\theta\otimes R
\end{equation}
for some 1-form $\theta$.  
The aim of the paper is to give a description of such structures. 
 Since the Lorentzian recurrent spaces are classified in \cite{WalkerRec},
 we will focus on the non-closed Weyl structures. 

In Section \ref{sec} we find the holonomy algebra of a general non-closed recurrent Lorentzian Weyl manifold. In particular, for $\dim M\geq 4$ we have the following statement. 

\begin{prop} \label{prop:distr}
Let $(M,c,\nabla)$ be a non-closed recurrent Lorentzian Weyl manifold. Suppose that $M$ is simply connected and let $\dim M =n+2\geq 4$. Then $M$ admits a uniquely defined $\nabla$-parallel distribution of space-like lines $\Ll$. The orthogonal distribution $\Ll^\bot$ is parallel as well and there exists a parallel isotropic vector field $p$ tangent to $\Ll^\bot$. The vector field $p$ is defined up to a constant multiple.  
\end{prop}

We show that there is a preferred metric in the conformal class defined up to homothety.

\begin{prop} \label{Propmetrh}
	Let $(M,c,\nabla)$ be a non-closed recurrent Lorentzian Weyl manifold and  $\dim M=n+2\geq 4$. 	Suppose that the manifold $M$ is simply connected. Then there exists a metric $h\in c$ satisfying 
	 \begin{equation}
	\label{tildeg}
	\nabla h=-2\omega_h\otimes h,\quad \nabla R=-3\omega_h\otimes R.\end{equation}
		The metric $h$ is defined uniquely up to homothety.
	\end{prop}

Following \cite{BHMM}, we consider the density bundle of weight $w$,
$$ L^w=P_{CO}\times_{|\det|^{\frac{w}{n+2}}}\Real,$$
where $P_{CO}$ is the conformal frame bundle.  According to \cite{BHMM}, each metric $g\in c$ defines a non-vanishing section  $$l_g\in  L^1$$  satisfying $$\nabla l_g=\omega_g\otimes l_g.$$  
 We conclude that $$\nabla (R\otimes l^{3}_{h})=0,$$ i.e., the tensor field $R$ is weighted parallel.

\begin{cor}
	If the curvature tensor of a non-closed Weyl  Lorentzian structure is recurrent, then it is weighted parallel. 
\end{cor}

The non-closed Weyl manifolds with weighted parallel curvature tensors	provide generalization of pseudo-Riemannian symmetric spaces. Another approach to symmetric spaces in conformal geometry may be found in~\cite{GZ}.

In Section \ref{sec:local} we use the holonomy algebras found in Section \ref{sec} and results from \cite{ParSpin} to obtain a local form of the general non-closed recurrent Weyl manifold $(M,c,\nabla)$.

\begin{theorem}\label{ThFinalLocal} 	Let $(M,c,\nabla)$ be a non-closed recurrent Lorentzian Weyl manifold and  $\dim M=n+2\geq 4$. Then in a  neighborhood of each point of $M$ the local conformal class contains the metric  
	$$h=(dt)^2+e^{2G}\left(2dvdu+\sum_{i=1}^{n-1}(dx^i)^2\right)$$
	where 
	\begin{equation}
	\label{funG}G=-\ln|u+\psi(t)|+\frac{1}{2}\ln|\psi'(t)|,\end{equation}
	$\psi(t)$ is a function with non-vanishing $\psi'(t)$, and 
	   $t,v,x^1,\dots,x^{n-1},u$ are local coordinates. 
	The 1-form corresponding to $h$ is given by
	$$\omega_h=-\partial_tGdt.$$ The pair $(h,\omega_h)$ satisfies \eqref{tildeg}.
\end{theorem}

It is easily checked that the local metric  $g=e^{-2G}h$ is flat, i.e., the conformal class $c$ is flat, and we obtain the following proposition. 

\begin{prop}\label{Propcisflat}
	All non-closed recurrent Lorentzian Weyl structures on manifolds of dimension greater than 3 are conformally flat. 
\end{prop}

By Theorem \ref{ThFinalLocal}, a recurrent Lorentzian Weyl structure is locally defined by a function $\psi(t)$. In \eqref{funG} we require $|\psi'(t)|>0$, which implies that either $\psi'(t)>0$ or $\psi'(t)<0$. The transformation $( \tilde v, \tilde u, \tilde x^i, \tilde t)= (-v,-u,x^i,t)$ preserves the coordinate form of the metric, but with $\tilde \psi(\tilde t)=-\psi(t)$. Since $\frac{d}{d\tilde t} \tilde \psi(\tilde t) = - \frac{d}{dt} \psi(t)$, a Weyl structure given by $\psi(t)$ with $\psi'(t)<0$ is equivalent to the one given by $\tilde \psi(\tilde t)=-\psi(t)$, which satisfies $\tilde \psi'(\tilde t)>0$. By this reason we may assume that $\psi'(t)>0$, which lets us replace $|\psi'(t)|$ in \eqref{funG} with $\psi'(t)$ without loosing any Weyl structures. Notice that the domain of definition is transformed under this coordinate transformation. If the Weyl structure satisfies $u+\psi(t)>0$, then it follows that $\tilde u + \tilde \psi(\tilde t) <0$. Similarly, $u+\psi(t)<0$ implies $\tilde u + \tilde \psi(\tilde t) >0$. 

Section \ref{sect:coord} is devoted to the investigation of the remaining coordinate freedom. In Proposition \ref{psl2} we find the vector fields preserving the coordinate form of Theorem \ref{ThFinalLocal}. By integrating the vector fields we see that the local equivalence problem of non-closed recurrent Lorentzian Weyl manifolds of dimension $\geq 4$ can be reduced to the local equivalence problem of the space $\{\psi \in C^\infty (\mathbb R) \mid \psi'(t)>0\}$ with respect to the action \eqref{eq:groupactionD4} of the Lie group $\SAff \times \PSL \times \mathbb Z_2$.

\begin{theorem} Two local recurrent Lorentzian Weyl structures defined by functions $\psi_1$ and $\psi_2$ with $\psi_i'(t)>0$ for $i\in\{1,2\}$ are equivalent if and only if 
	\[\psi_2(t)= \mathrm{sgn}(B)  \frac{a \psi_1(A+ Bt) + b}{c \psi_1(A+Bt) +d},\]
	for some constants $a,b,c,d,A,B$ satisfying $ad-bc=1, B \neq 0$. 
\end{theorem}  	

The local equivalence problem is solved for generic non-closed recurrent Lorentzian Weyl structures in Section \ref{sect:diffinv} by showing (in Proposition \ref{Prop:invariants}) that the differential invariants are generated by the two differential invariants
 \[  \frac{(\psi_1^2 \psi_4-4 \psi_1 \psi_2 \psi_3+3 \psi_2^3)^2}{(2 \psi_1 \psi_3-3 \psi_2^2)^3}, \qquad \frac{\psi_1 (\psi_1^2 \psi_5-5 \psi_1 \psi_2 \psi_4+5 \psi_2^2 \psi_3)}{(2 \psi_1 \psi_3-3 \psi_2^2)^2},\]
where $\psi_1, \cdots, \psi_5$ are coordinates on the space of $5$-jets of $\psi(t)$.

Suppose that the manifold $M$ is simply connected.  Let us fix a metric $h \in c$ as in Proposition~\ref{Propmetrh}. Since the manifold $M$ is simply connected,  there exists a non-vanishing section $X$ of the distribution $\Ll$ from Proposition \ref{prop:distr}. Let us fix a section $X$ satisfying $$h(X,X)=1.$$
Note that if $h$ is not fixed, then the section $X$ is defined uniquely up to multiplication by a non-zero constant. 
In Section \ref{sec:global} we prove the following global result.

\begin{theorem} \label{Thglobal}	Let $(M,c,\nabla)$ be a non-closed recurrent Lorentzian Weyl structure and  $\dim M=n+2\geq 4$. Suppose that the manifold $M$ is simply connected. Fix a metric  $h \in c$ as in Proposition~\ref{Propmetrh}. Suppose that the vector field $X$ is complete. Then the manifold $M$ is diffeomorphic to the product
	$$M\cong \Real\times N,$$
	where $N$ is an integral manifold of the $\nabla$-parallel distribution orthogonal to $X$. The distribution $p^\bot$ on $N$ orthogonal to the $\nabla$-parallel isotropic vector field $p$ is involutive. The corresponding codimension-one foliation ${\mathcal F}$ on $N$ admits a transversal projective structure. There is a foliated atlas on $N$ such that to each chart $U$ with coordinates $v,x^1,\dots ,x^{n-1}, u$, where $u$ is the transversal coordinate,  corresponds a regular function $\psi\in C^\infty(\Real)$ such that on $\Real\times U$, $h$ and $\omega_h$ are as in Theorem \ref{ThFinalLocal}. Given two such charts $U_1$ and $U_2$ with non-trivial intersection, the transverse coordinates are transposed as $$u_2=\frac{au_1+b}{cu_1+d},\quad ad-bc=1$$ and the corresponding functions satisfy
	$$\psi_2=\frac{a\psi_1-b}{-c\psi_1+d}.$$
\end{theorem}

A local diffeomorphism $\varphi$ of $M$ with domain $U \subset M$ is a local symmetry of the recurrent Weyl structure $(c, \nabla)$ if it preserves the conformal structure $c$ and the affine connection $\nabla$. In the case that the Weyl structure is defined by a metric $g \in c$ and a 1-form $\omega_g$, the condition on $\varphi$ may be expressed as $\varphi^* g=e^{2f} g$ and $\varphi^* \omega_g = \omega_g-df$ for some function $f \in C^\infty(U)$. If $U=M$, then $\varphi$ is a global symmetry.  A vector field $X$ on $U \subset M$ is called a local infinitesimal symmetry of the recurrent Weyl structure on $M$ defined by $(g,\omega_g)$ if it satisfies $L_X g = 2\lambda g$ and $L_X \omega_g = -d \lambda$ for some function $\lambda \in C^\infty(U)$, and a global infinitesimal symmetry if the same holds for $U=M$.  We will often drop the adjective ``infinitesimal'' when it is clear that the symmetry referred to is a vector field. 

The generic non-closed recurrent Lorentzian Weyl manifold of dimension $n+2\geq 4$ has $(2n-1)+\binom{n-1}{2}$ independent infinitesimal symmetries, and they span a Lie algebra abstractly isomorphic to $\mathfrak{heis}(n-1) \rtimes \mathfrak{so}(n-1)$, but there are some non-closed recurrent Lorentzian Weyl manifolds with additional infinitesimal symmetries. These are studied in Section \ref{sect:symmetric}. We show that there is exactly one class of locally homogeneous spaces and 5 classes of cohomogeneity-one spaces. All other spaces have cohomogeneity two. More precisely, we obtain the following result.


\begin{theorem} Let $(M,c,\nabla)$ be a non-closed recurrent Lorentzian Weyl manifold and  $\dim M=n+2\geq 4$. The following statements hold in the neighborhood of any point. 
	\begin{itemize}
		\item[1.] The infinitesimal symmetries span the tangent space 
		at each point of $M$ if and only if the representative pair $(h,\omega)$ can be given in the local coordinates of Theorem \ref{ThFinalLocal} by the function $\psi(t)=t$. 
		\item[2.] The infinitesimal symmetries span a hyperspace of the tangent space 
		at each point of $M$ if and only if $(h,\omega)$ can be given in the local coordinates of Theorem \ref{ThFinalLocal} by a function $\psi$ taking one of the following forms:
\begin{align*}
 \psi(t) &= e^{t}, \\ 
\psi(t) &= \tan(t), \\
 \psi(t)&=A \ln t, \quad A > 0,  \\
 \psi(t) &= \tan(A \ln t), \quad A>0, \\
   \psi(t) &= |t|^{A}, \quad  At>0, \; A \neq 1.
\end{align*}
		\item[3.] In all other cases, the infinitesimal symmetries span a subspace of codimension 2 of the tangent space 
		at each point of $M$.
	\end{itemize}
\end{theorem}

Denote by $\operatorname{Aut}(M,c,\nabla)$ the  group of global symmetries and by $\operatorname{Aut}^0(M,c,\nabla)$ the connected component of the identity.  In Section \ref{sect:homogeneous} we show that in each dimension  $n+2\geq 4$ there is only one simply connected homogeneous non-closed recurrent Lorentzian Weyl manifold.
Define the Lie group
$$G_0=\Real^+\zr (\Real\oplus\Real^{n-1}\oplus\Real)$$
with the operation
$$(r_1,w_1,y_1,s_1)\cdot (r_2,w_2,y_2,s_2)=(r_1r_2,w_1+r_1^{3}w_2,y_1+r_1^{2}y_2,s_1+r_1s_2).$$
Let $b$ be the left-invariant Lorentzian metric on $G_0$ with the value
$$b_e= \left(4(dr)^2+2dwds +\sum_{i=1}^{n-1}(dy^i)^2+2(ds)^2\right)_e$$
at the identity $e=(1,0,0,0)\in G_0$. Consider the left-invariant 1-form $\omega$
defined by the value $$\omega_e=\left(-\frac{\sqrt{2}}{2}ds+2dr\right)_e.$$ We show that the Weyl connection $\nabla$ defined by $b$ and $\omega$ has recurrent curvature tensor.

\begin{theorem}\label{ThHomog}
	Let $(M,c,\nabla)$ be a non-closed recurrent Lorentzian Weyl structure and  $\dim M=n+2\geq 4$. Suppose that the manifold $M$ is simply connected and $\operatorname{Aut}(M,c,\nabla)$-homogeneous.
	Then $M$ is diffeomorphic to $G_0$. Under this diffeomorphism,
	$c=[b]$, and $\nabla$ is defined by $b$ and $\omega$. Moreover,
	\begin{equation}\label{Aut=}\operatorname{Aut}^0(M,c,\nabla)=\operatorname{Isom}^0(G_0,b)\cong G_0\cdot\SO(n-1)\cdot\Real^{n-1}.\end{equation}
\end{theorem} 

Several of the results mentioned so far (concerning local coordinate expressions, symmetries and differential invariants) have analogues for $\dim M=3$, and they can be found in respective sections together with the results for $\dim M \geq 4$.  

A special type of Lorentzian Weyl manifolds that has received a lot of attention (see for example \cite{BFKN, DMT}) are the Einstein-Weyl manifolds, and in Section \ref{sect:EinsteinWeyl} we investigate whether some these manifolds coincide with the non-closed recurrent Weyl manifolds. For $\dim M \geq 4$ it turns out that there are no Einstein-Weyl manifolds among the non-closed recurrent Lorentzian Weyl manifolds. However, there does exist non-closed Lorentzian Einstein-Weyl manifolds of dimension 3 with recurrent curvature tensor. By Proposition \ref{prop:EW} they are exactly the nonclosed recurrent Lorentzian Weyl manifolds with 2-dimensional holonomy algebra, and they can be locally defined by the pair $(g,\omega)$ given by 
\[ g= 2dv du +(dx)^2 +\left (a(u) v x + \frac{1}{12} a^2(u) x^4-\frac{1}{3} \dot a(u) x^3 + c(u) x\right) (du)^2, \qquad \omega= a(u) x du,\]
where the function $a$ is non-vanishing. As shown in Section \ref{sect:symmetric}, the most symmetric of these have a 2-dimensional Lie algebra of infinitesimal symmetries, and they are locally equivalent to the one defined by 
\[ a(u) \equiv 1, \qquad c(u) \equiv 0,\] 
which admits the infinitesimal symmetries
\[\partial_u, \qquad u\partial_u-3v\partial_v-x\partial_x.\]

Note that the non-closed recurrent Lorentzian Weyl structures  have another remarkable property: they admit a rather large number of weighted parallel spinors~\cite{ParSpin}.

\begin{cor}
	\label{cor2}
	Let $(c,\nabla)$ be a non-closed recurrent Lorentzian Weyl spin structure on a simply connected manifold $M$ of dimension $n+2\geq 4$, then it admits a space of weighted parallel spinors of complex dimension $2^{[\frac{n}{2}]}$. 
\end{cor}

A similar result is known for Lorentzian symmetric spaces \cite{L07}: each simply connected  indecomposable Cahen-Wallach spin manifold of dimension $n+2$ admits a space of parallel spinors of complex dimension $2^{[\frac{n}{2}]}$ (Cahen-Wallach spaces are symmetric Lorentzian spaces represented by a special class of pp-waves).

Finally note that the condition \eqref{nablaR} may be generalized in many various ways, see the recent survey~\cite{Sen24}.

\section{The holonomy algebra of  recurrent Lorentzian Weyl spaces}\label{sec}

We start by recalling the classification of the holonomy algebras of Lorentzian Weyl structures~\cite{Andr}.

Let $(M,c,\nabla)$ be a Weyl manifold of Lorentzian signature $(1,n+1)$, $n\geq 1$. Then its holonomy algebra is contained in the conformal Lorentzian algebra $$\co(1,n+1)=\Real\id_{\Real^{1,n+1}}\oplus\so(1,n+1).$$ The Weyl structure is closed if and only if the holonomy algebra is contained in $\so(1,n+1)$ and this case is well-studied \cite{L07}. By that reason we suppose that the Weyl structure is non-closed, and the holonomy algebra is not contained in  $\so(1,n+1)$.

Fix a Witt basis $p,e_1,\dots,e_n,q$ of the Minkowski space $\Real^{1,n+1}$. With respect to that basis the subalgebra of $\so(1,n+1)$ preserving the null line $\Real p$ has the following matrix form:

\begin{equation*}
\frakso (1, n + 1)_{\mathR p} = \left\{
\left. 
\begin{pmatrix}     
a & X^t & 0 \\
0 & A & -X \\
0 & 0 & -a
\end{pmatrix} \right|
\begin{matrix}     
a \in \mathR \\
A \in \frakso (n) \\
X \in \mathR^n
\end{matrix} \right\} .
\end{equation*}	
We identify the Lie algebra $\so(1,n+1)$ with the space of bivectors $\wedge^2\Real^{1,n+1}$ in such a way that
$$(X\wedge Y)Z=(X,Z)Y-(Y,Z)X.$$
Under this identification  the above element of  
$\frakso (1, n + 1)_{\mathR p}$ corresponds to
$$-a p\wedge q + A - p\wedge X.$$
We get the decomposition 
$$\frakso (1, n + 1)_{\mathR p} = (\mathbb{R}p\wedge q\,\oplus\frakso(n))\ltimes\,p\wedge\mathbb{R}^n.$$

The holonomy algebra $\g\subset\co (1, n + 1)_{\mathR p}$ of a non-closed Lorentzian Weyl structure of dimension $n+2$, $n\geq 1$, satisfies one of the following 3 conditions:

\begin{itemize}
	\item[\bf 1.]  $\g\subset\co (1, n + 1)$ is irreducible. In this case $\g=\co (1, n + 1)$.	
	
	\item[\bf 2.]  $\frakg\subset  \frakco (1, n + 1)$  preserves  an orthogonal decomposition  $$\mathR^{1, n + 1} = \mathR^{1, k + 1} \oplus \mathR^{n - k}, \quad -1 \leqslant k \leqslant n - 1.$$ In this case   $\frakg$ is one of the following:
	\begin{itemize}
		\item[$\bullet$] $\mathR \id_{\Real^{1,n+1}}\oplus \frakso (1, k + 1) \oplus \frakso (n - k) $, $ \quad -1 \leqslant k \leqslant n-1$;
		\item[$\bullet$] $ \mathR (\id_{\Real^{1,n+1}}+p\wedge q) \oplus \mathfrak{k} \oplus \frakso (n - k) \ltimes p\wedge \mathR^{k}\subset\frakco (1, n + 1)_{\mathR p}$,  $0\leq k\leq n-1$;
		\item[$\bullet$] $\mathR \id_{\Real^{1,n+1}}\oplus\mathR p\wedge q \oplus \mathfrak{k} \oplus \frakso (n - k) \ltimes p\wedge \mathR^{k}\subset\frakco (1, n + 1)_{\mathR p}$,  $1\leq k\leq n-1$.
	\end{itemize}
	Here $\mathfrak{k} \subset \frakso (k)$ is the holonomy algebra of a Riemannian manifold.
		
	\item[\bf 3.] $\g$ is contained in $\frakco (1, n + 1)_{\mathR p}$, and it does not preserve any proper non-degenerate subspace of $\Real^{1,n+1}$. Such algebras may be divided into 6 types. For the results of the current paper it is enough to know that each such $\g$ contains the ideal $p\wedge \mathR^{n}$.
\end{itemize}  

Note that the algebras satisfying the second condition correspond to conformal product in the sense of \cite{Moroianu1}.

\begin{theorem}\label{Thhol} Let $(M,c,\nabla)$ be a recurrent non-closed Lorentzian Weyl manifold of dimension $n+2\geq 3$, then one of the following conditions holds: 	
	\begin{itemize}
	\item $n=1$ and the holonomy algebra of $\nabla$  is one of the following
	$$\g = \mathR (\id_{\Real^{1,2}} + p\wedge q),\quad \mathR (2 \id_{\Real^{1,2}} + p\wedge q) \ltimes \mathR p\wedge e_1; 
	$$
	\item $n\geq 2$ and the holonomy algebra of $\nabla$ is
	$$\mathR (\id_{\Real^{1,n+1}}+p\wedge q)  \ltimes p\wedge \mathR^{n-1}\subset\frakco (1, n + 1)_{\mathR p}.$$
	\end{itemize}	
\end{theorem}

{\bf Proof of Theorem \ref{Thhol}.} Let $(M,c,\nabla)$ be a recurrent non-closed Lorentzian Weyl manifold and let $\g\subset\co(1,n+1)$ be its holonomy algebra at a point $x\in M$. Let us denote the tangent space $T_xM$ by $\Real^{1,n+1}$. Suppose that $R_x\neq 0$.
Since $R$ is recurrent, from the Ambrose-Singer Theorem it follows that
\begin{equation}
	\label{sv1}
	\g={\rm span}\{R_x(X,Y)|X,Y\in \Real^{1,n+1}\}.
\end{equation}
Consider the space of algebraic tensors of type $\g$:
\begin{align*}
{\mathcal R}(\g)=\bigl \{ R \in \operatorname{Hom} (\wedge^2 \Real^{1,n+1},\mathfrak{g}) \mid{}
&R(X,Y)Z+R(Y,Z)X+R(Z,X)Y=0\quad \forall X,Y,Z\in \Real^{1,n+1} \bigr \}.
\end{align*}
From the Bianchi identity it follows that $R_x\in {\mathcal R}(\g)$. 
The holonomy algebra $\g$ acts on the space  ${\mathcal R}(\g)$ in the following natural way
\begin{equation}
\label{eqxiR}
\xi:R\mapsto \xi\cdot R,\quad (\xi\cdot R)(X,Y)= [ \xi,  R(X,Y)] - R(\xi X,Y) - R(X, \xi Y).
\end{equation}
Since the curvature tensor $R$ is recurrent, by the holonomy principle, $\g$ preserves the line $\Real R_x$,
i.e., there exists a 1-form $\rho$ on $\g$ such that
\begin{equation}
	\label{sv2}
	\xi\cdot R_x=\rho(\xi)R_x\quad\forall\, \xi\in\g.
\end{equation}
Now we are going to describe the holonomy algebras $\g$ of non-closed Weyl connections and elements $R\in {\mathcal R}(\g)$ satisfying the properties \eqref{sv1} and \eqref{sv2}. 
Consider the 3 possibilities for $\g$ from above.

{\bf Case 1.} Suppose that $\g=\co (1, n + 1)_{\Real p}$.
There exists an isomorphism of $\so (1, n + 1)$-modules 
$${\mathcal R}(\co (1, n + 1))\cong {\mathcal R}(\so (1, n + 1))\oplus\, \wedge^2\Real^{1, n + 1},$$
where an element $A \in \wedge^2\Real^{1, n + 1}$ defines an algebraic curvature tensor $R_A$ by the equality
\begin{equation}
\label{eqRA}
R_A (X, Y) = AX \wedge Y + X \wedge A Y + 2(AX, Y) \id_{\Real^{1,n+1}}.
\end{equation}

Let $R\in{\mathcal R}(\co (1, n + 1)).$ Property \eqref{sv2} implies that $R\in {\mathcal R}(\so (1, n + 1))$. Such $R$ does not satisfy property \eqref{sv1}.
 
{\bf Case 3.} Suppose that $\g\subset\frakco (1, n + 1)_{\mathR p}$. Then $\g$ contains the ideal $p\wedge \Real^n$. First suppose that $n\geq 2$,
Let $\xi = p\wedge Z$, $Z \in \Real^n$. Let 
$R = R_0 + R_A$, where $R_0 \in {\mathcal R}(\so (1, n + 1))$ and an element $A \in \wedge^2\Real^{1, n + 1}$ defines $R_A$ by~\eqref{eqRA}. The element $A$ may be rewritten in the form
$$A=a p\wedge q+p\wedge X+q\wedge Y+B,\quad a\in\Real,\,\,X,Y\in\Real^n,\,\, B\in\so(n).$$
The equality~\eqref{sv2} implies
$$[\xi, A] = \rho(\xi) A,$$
or, equivalently,
$$aZ-BZ=\rho(\xi)X,\quad Z\wedge Y=\rho(\xi)B,\quad \rho(\xi)a=0,\quad \rho(\xi)b=0.$$
If $\rho(\xi) \neq 0$ for some $\xi\in p\wedge\Real^n$, then $A = 0$, and  $R$ does not satisfy property~\eqref{sv1}. This implies that $\rho(\xi) = 0$ for all $\xi\in p\wedge\Real^n$. Consequently, $a=0$, $B=0$, $Y=0$, i.e., $A=p\wedge X$. Combining this with the description of the element $R\in\frakco (1, n + 1)_{\mathR p}$ given in \cite{Andr}, we conclude that 
 $R$ is determined by the equalities
\begin{align*}
&R(p, q) = -\lambda p\wedge q- p \wedge X_0, \quad R(p, V) = 0,\\
&R(U, V) = - p \wedge \big( P(V) U - P(U) V -2g(V,X) U +2 g(U,X) V \big) , \\
&R(U, q) =  \gamma (U) \id  - g(U, 2X+X_0)  p \wedge q + P(U) - p \wedge K(U), 
\end{align*}
where
$$ 
\lambda \in \mathbb{R}, \quad
X_0,  \in \mathbb{R}^n, \quad
P \in \Hom(\mathbb{R}^n, \so(n)), \quad
K \in \odot^2 \mathbb{R}^n, \quad
S  \in {\mathcal R} (\mathfrak{so} (n))
$$
are fixed, and $U,V\in\Real^n$ are arbitrary. 
The condition $p\wedge\Real^n\,\cdot R=0$ implies  $\lambda = 0$, $S = 0$, and 
$$2 P(V) U - P(U) V  +2g(U,X)V+ g(V, X_0) U + (U, V) X_0=0 \quad \forall \, U, V \in \Real^n.$$ Using this equality and the similar equality, where $U$ is interchanged with $V$, we get
$$3P(V)U+g(U,4X+X_0)V+2g(V,X+X_0)V+3g(U,V)X_0=0 \quad \forall \, U, V \in \Real^n.$$ The condition $P(V)\in\so(n)$ implies $g(P(V)V,W)=-g(P(V)W,V)$ for all $V,W\in\Real^n$. Using this, we get $X=-X_0$, and $P(V)=-V\wedge X_0$. Thus it holds 
\begin{align}
\label{algReq2}
\begin{split}
&R(p, q) = - p \wedge X_0, \quad R(p, V) = 0,\\
&R(U, V) = - p \wedge \big( 3 (U \wedge V) X_0 \big) , \\
&R(U, q) =  -2 g(U, X_0) \id - 3 g(U, X_0) p \wedge q - U\wedge X_0 - p \wedge K(U).
\end{split}
\end{align}
The condition \eqref{sv1} implies $X_0\neq 0$. We may suppose that $X_0=e_1$. Since 
$$R(p,q) = - p \wedge e_1, \quad R(e_i,e_1) = 3 p \wedge e_i, \quad R(e_i,q) = - e_i \wedge e_1 - p \wedge K(e_i), \quad i > 1,$$
it holds  $- e_i \wedge e_1 \in \g$. From \eqref{eqxiR} and~\eqref{sv2} for $\xi = - e_i \wedge e_1$, $X = p$, $Y = q$ it follows that
$$(\xi\cdot R)(p,q) = - p \wedge e_i = - \rho (\xi)  p \wedge e_1,$$
which gives a contradiction, i.e., it is impossible that $\g$ is contained in $\frakco (1, n + 1)_{\mathR p}$ and $n\geq 2$.

\medskip

Now suppose that $n = 1$. Results of~\cite{Andr} show that $\frakg$ is  one of the following:
\begin{itemize}
	\item[\bf a.] $\mathR \id_{\Real^{1,2}} \oplus \mathR p\wedge q \ltimes  \mathR p\wedge e_1$;
	
	\item[\bf b.] $\mathR \id_{\Real^{1,2}} \ltimes  \mathR p\wedge e_1$;
	
	\item[\bf c.] $\mathR (\alpha \id_{\Real^{1,2}} + p\wedge q) \ltimes \mathR p\wedge e_1$, $\ \alpha \in \mathR$.
\end{itemize}  

Let $\xi = p\wedge e_1$. If $\rho(p\wedge e_1) \neq 0$, then it is easy to show that the Weyl structure is closed. If $\rho(p\wedge e_1) = 0$, then as above we obtain  the following equalities for $R$:
\begin{equation}
	\label{algReq1}
	R(p,q) = - \beta p\wedge e_1, \quad R(p,e_1) = 0, \quad R(e_1, q) = \beta (2 \id_{\Real^{1,2}} + p\wedge q) - k p\wedge e_1, \quad \beta, k \in \Real.
\end{equation}
This and the property~\eqref{sv1} imply  that $\g$ is the algebra from point c. Moreover, it holds $\alpha = 2$, $\rho(2\id + p\wedge q) = -5$, and $k = 0$.
 
{\bf Case 2.} First consider the holonomy algebra $\g=\mathR \id_{\Real^{1,n+1}}\oplus \frakso (1, k + 1) \oplus \frakso (n - k) $, $-1 \leqslant k \leqslant n-1$. In \cite{Andr}, it is shown that 
$${\mathcal R}(\g)\cong {\mathcal R}(\frakso (1, k + 1)) \oplus {\mathcal R}(\frakso (n - k))\oplus \, \mathR^{1, k + 1} \otimes \mathR^{n - k}.$$
There exists an invariant line in the $\frakso (1, k + 1) \oplus \frakso (n - k)$-module $\mathR^{1, k + 1} \otimes \mathR^{n - k}$
if and only if $1\leq k+2\leq 2$ and $n-k=1$.
Hence, for $n\geq 2$ there is no $R\in{\mathcal R}(\g)$ satisfying \eqref{sv1} and \eqref{sv2}. 
The only possible situation is $k=0$, $n=1$ and
$$\g = \mathR \id_{\Real^{1,2}}\oplus \frakso (1, 1) \oplus \frakso (1) \cong \mathR \id_{\Real^{1,2}} \oplus \mathR p\wedge q \subset\frakco (1, 2)_{\mathR p}.$$
According to~\cite{Andr}, $R$ is uniquely determined by the following equalities:
$$
R(p, q) = - \lambda p \wedge q, \quad R(p, e_1) = 0,\quad
R(e_1, q) = \gamma (\id_{\Real^{1,2}} + p \wedge q), \quad
\lambda,\gamma \in \Real.$$
From~\eqref{sv2} for $\xi = a\id_{\Real^{1,2}}$, $a \in \Real$ it follows that $\rho(a\id_{\Real^{1,2}}) = -2a$. Next, we use~\eqref{eqxiR} and~\eqref{sv2} for $\xi = b p \wedge q$, $b \in \Real$. Substituting $X = e_1$, $Y = q$, we obtain $\rho(b p \wedge q) = -b$.
Also, for $X = p$, $Y = q$ we have $\rho(b p \wedge q) R(p,q) = 0$ for all $b \in \Real$; hence $\lambda = 0$ and $R$ does not satisfy property~\eqref{sv1} ($\g$ does not contain $\Real p \wedge q$).

Suppose that 
$$\g=\mathR (\id_{\Real^{1,n+1}}+p\wedge q) \oplus \mathfrak{k} \oplus \frakso (n - k) \ltimes p\wedge \mathR^{k}\subset\frakco (1, n + 1)_{\mathR p}, \quad 0\leq k\leq n-1.$$
According to \cite{Andr}, each $R\in {\mathcal R}(\g)$ may be written in the form
$R = R_1 + R_2 + R_3$, where $R_1 \in {\mathcal R}(\mathfrak{k} \ltimes p\wedge \mathR^{k}),$ $R_2 \in {\mathcal R}(\frakso (n - k))$, and after a proper choice of the basis it holds
\begin{align}
\label{R3eq1}
\begin{split}
&R_3 (e_i, e_{k+1}) = a p \wedge e_i, \quad 1 \leqslant i \leqslant k,\\
&R_3 (e_{k+1}, q) = -a (\id_{\mathbb{R}^{1,n+1}} + p \wedge q),\\
&R_3 (e_j, q) = - a e_{k+1} \wedge e_j, \quad k + 2 \leqslant j \leqslant n
\end{split}
\end{align}
for some $ a \in \mathbb{R}$.
If $n - k \geqslant 2$ and $j > k + 1$ then $e_{k+1} \wedge e_j \in \mathfrak{g}$ and from~\eqref{sv2} we get
\begin{align*}
((e_{k+1} \wedge e_j) \cdot R_3) (e_j, q) = -a [e_{k+1} \wedge e_j, e_{k+1} \wedge e_j] + &R_3 (e_{k+1}, q) =\\
= &R_3 (e_{k+1}, q) = \rho (e_{k+1} \wedge e_j) R_3 (e_j, q). 
\end{align*}
From the last equation it follows that $a = 0$, $R_3 = 0$ and $R$ does not satisfy property~\eqref{sv1}. We conclude that $n-k = 1$ and
$$\g=\mathR (\id_{\Real^{1,n+1}}+p\wedge q) \oplus \mathfrak{k} \ltimes p\wedge \mathR^{n-1}.$$
Using arguments as in Case 1, it is not hard to show that $\mathfrak{k}=0$, and it holds $R=R_3$, where the non-zero values of $R_3$ are the following:
\begin{equation}
\label{R13eq1}
R_3 (U, e_n) = a p \wedge U,\quad 
R_3 (e_n, q) = -a (\id_{\mathbb{R}^{1,n+1}} + p \wedge q)
\end{equation}
for all $ U \in \mathR^{n-1}$.
 
Finally applying the arguments we used just above, it is easy to check that the holonomy algebra 
$$\g=\mathR \id_{\Real^{1,n+1}}\oplus\mathR p\wedge q \oplus \mathfrak{k} \oplus \frakso (n - k) \ltimes p\wedge \mathR^{k}\subset\frakco (1, n + 1)_{\mathR p}, \quad 1\leq k\leq n-1$$
does not satisfy the required properties. This proves the theorem. \qed

\section{The local form of recurrent Weyl structures}\label{sec:local}

In what follows the dot over the function denotes the partial derivative in the direction of the coordinate $u$.

\begin{theorem}
	\label{MainTh}
	Let $(M,c,\nabla)$ be a  non-closed Lorentzian Weyl manifold of dimension $n+2\geq 4$ with recurrent curvature tensor. Then  around each point of $M$ there exist coordinates $v,x^1,\dots,x^n,u$ such that the class $c$ is locally represented by the metric
	\begin{equation}\label{metricag}g=2dvdu+\sum_{i=1}^{n-1}(dx^i)^2+e^{-2F}(dx^n)^2+a(u)\sum_{i=1}^{n-1} (x^i)^2 (du)^2,\end{equation} 
	the corresponding 1-form $\omega$ is given by	
	$$\omega=\dot Fdu,$$ where $F=F(x^n,u)$ is a function such that $\dot F$ satisfies the Riccati  
	\begin{equation}\label{eqmain}\ddot F-\dot F^2=-a(u)\end{equation} and such that $\partial_n\dot F$ is non-vanishing.
\end{theorem}

It is remarkable that the Riccati equation is also related to the Einstein-Weyl equation in dimension~2~\cite{Cal1,Cal2}. Similarly, the Einstein-Weyl equation in Lorentzian signature and dimension 3 is equivalent to the dKP equation, and that has led to many interesting results~\cite{BFKN,DP,DGS,DMT}.

{\bf Proof of Theorem \ref{MainTh}.}
Let $(M,c,\nabla)$ be a non-closed Lorentzian Weyl manifold of dimension $n+2\geq 4$ with recurrent curvature tensor.  By~Theorem~\ref{Thhol} its holonomy algebra coincides with
$$\mathR (\id_{\Real^{1,n+1}}+p\wedge q)  \ltimes p\wedge \mathR^{n-1}\subset\frakco (1, n + 1)_{\mathR p}.$$
From~\cite[Theorem 8]{ParSpin} it follows that around each point of $M$ there exists a coordinate neighborhood $U$ with  coordinates $v,x^1,\dots , x^n, u$ and a~metric $g \in c$ such that
$$g=2dvdu+h+H(du)^2,$$
 $$h=\sum_{i,j=1}^{n-1}\delta_{ij}dx^idx^j+e^{-2F}(dx^n)^2,$$
$$\omega =fdu,\quad f = \dot{F},$$
where $H = H(x^1,\dots , x^{n-1},u)$ and $F = F(x^n,u)$ are  functions. 
The Christoffel symbols of the connection $\nabla$ are the following:
\begin{align*}
\label{GammaEq}
\Gamma_v &= 0,\\
\Gamma_i &= \begin{pmatrix} 
0 & -(\delta_{ik} \dot{F})^{n}_{k = 1} & \frac{1}{2} \partial_i H \\
0 & 0 & ((\delta_{ik} \dot{F})^{n}_{k = 1})^{t} \\
0 & 0 & 0 \end{pmatrix}, \quad i = 1,\ldots, n - 1 \\
\Gamma_n &= (- \delta_{bn} \delta_{cn} \partial_n F)_{b,c = v,1,\ldots , n, u}, \\
\Gamma_u &= \begin{pmatrix}
0 & (\frac{1}{2} \partial_k H)^{n-1}_{k = 1} & 0 & 0\\
0 & \dot{F} E_{n-1} & 0 & -((\frac{1}{2} \partial_k H)^{n-1}_{k = 1})^{t}\\
0 & 0 & 0 & 0\\
0 & 0 & 0 & 2 \dot{F} \end{pmatrix}.
\end{align*}
The components of the curvature tensor are as follows:
\begin{align*}
&R(\partial_v, \partial_u) = 0,\quad R(\partial_v, \partial_i) = 0, \quad R(\partial_v, \partial_n) = 0, \quad R(\partial_i, \partial_j) = 0, \quad i, j = 1,\ldots, n - 1,\\
&R(\partial_i, \partial_n) = \partial_n \dot{F} \begin{pmatrix} 
0 & (\delta_{ik})^{n}_{k = 1} & 0 \\
0 & 0 & -((\delta_{ik})^{n}_{k = 1})^{t} \\
0 & 0 & 0 \end{pmatrix},\\
&R(\partial_i, \partial_u) = \begin{pmatrix} 
0 & Z_{iu} & 0 \\
0 & 0 & -Z_{iu}^t\\
0 & 0 & 0 \end{pmatrix}, \quad \text{where} \quad Z_{iu} = \bigg(\bigg(\frac{1}{2} \partial_i \partial_k H + \delta_{ik} (\ddot{F} - \dot{F}^2)\bigg)^{n-1}_{k = 1}, 0\bigg)\\
&R(\partial_n, \partial_u) = \partial_n \dot{F} \begin{pmatrix} 
0 & 0 & -H \\
0 & 1 & 0 \\
0 & 0 & 2 \end{pmatrix}.
\end{align*}
Consider the field of frames  
$$p = \partial_v,\quad e_i = \partial_i, \quad q = \partial_u - \frac{1}{2} H \partial_v.$$
In notation of the proof of Theorem \ref{Thhol} it holds $R = R_3$, consequently  $R(\partial_i, \partial_u)$ must be equal to zero, in other words,
$$\frac{1}{2} \partial_i \partial_k H + \delta_{ik} (\ddot{F} - \dot{F}^2) = 0, \quad \text{for all} \quad i, k = 1,\ldots, n - 1.$$ 
From this we find out that
$$H(x^1,\dots , x^{n-1},u) = a(u)\sum_{i=1}^{n-1} (x^{i})^2 + \sum_{i=1}^{n-1} b_i(u) x^i + c(u),$$
where $a(u)$, $b_i(u)$, $c(u)$ are functions and
\begin{equation}
\label{Riccati1}
\ddot{F} - \dot{F}^2 = - a(u).
\end{equation} 
It is easy to see that the coordinates $v,x^1,\dots,x^{n-1}$ may be changed in such a way that $c(u)=0$ and all $b_i(u)=0$.
Now it is easy to check that the only non-zero component of $\nabla R$ is $\nabla_n R(\partial_i, \partial_n)$, and it holds 
$$\partial_n \dot{F}\cdot\nabla_n R(\partial_i, \partial_n) = \left(\partial^2_n \dot{F} + \partial_n F\cdot  \partial_n \dot{F}\right)R(\partial_i,\partial_n).$$
Let $\theta=\theta_vdv+\sum_{i=1}^n\theta_idx^i+\theta_udu$.
We get the equality 
$$\theta_n \partial_n \dot{F} = \partial^2_n \dot{F} + \partial_n F\cdot \partial_n \dot{F}.$$
Let $U_0\subset U$ be the open subset, where $\partial_n \dot{F}$ is non-vanishing. Suppose that $U_0\neq\emptyset$. Then
on $U_0$ it holds
$$\theta=\theta_n dx^n,\quad \theta_n=\partial_n F+\partial_n\ln|\partial_n\dot F|.$$
Since $\theta_n$ is a smooth function on $U$, $U_0=U$. Note that $d\omega=
\partial_n \dot{F} dx^n\wedge du$. We conclude that for each coordinate neighborhood $U$, either $d\omega|_U=0$ or $d\omega|_U$ is non-vanishing on $U$. Since $M$ is connected, we conclude that $d\omega$ and $\partial_n \dot F$ are non-vanishing on $M$. \qed

\begin{prop}\label{proptildeg} In the settings of the statement of Theorem \ref{MainTh}, the  conformal class of the local metric $g$ contains a unique up to homothety metric $h$  that satisfies  \eqref{tildeg}.
\end{prop}

\begin{proof}
	The proof of Theorem~\ref{MainTh} shows that the curvature tensor $R$
	satisfies $$\nabla R=\theta\otimes R,\quad \theta=\partial_n( F+\ln|\partial_n\dot F|) dx^n.$$
	It holds $d \theta= -3d\omega$, and
	$$\theta=-3\omega +d\varphi,\quad \varphi=F+\ln|\partial_n\dot F|.$$ 
	Consider the metric $$h=e^{\frac{2}{3}\varphi}g.$$ The corresponding 1-from is given by 
	$$ \omega_h=\omega-\frac{1}{3}d\varphi=-\frac{1}{3}\theta.$$
	Thus, $h$ satisfies \eqref{tildeg}. It is obvious that any other metric from the local conformal class satisfying \eqref{tildeg} is homothetic to $h$.
\end{proof}

\begin{prop}\label{lemchangeF} In the settings of Theorem~\ref{MainTh}, the coordinates and the local metric may be chosen in such a way that $a(u)=0$. Then 
	\begin{equation}\label{formulaf}F=-\ln| u+\psi(x^n)|+c( x^n)\end{equation}
	for some functions $\psi( x^n)$ and $c( x^n)$ such that $\psi'(x^n)$ is non-vanishing.
\end{prop}

\begin{proof}
Following \cite{KT} we consider the  coordinate transformation:
	\begin{align*} v&=\bar v-\frac{1}{2}\varphi'(\bar u)\sum_{i=1}^{n-1}(\bar x^i)^2,\\
	x^i&=e^{\varphi(\bar u)}\bar x^i,\quad i=1,\dots,{i-1},\\
	x^n&=\bar x^n,\quad
	u=u(\bar u),\end{align*}
	where $\varphi(\bar u)$ is a function and it holds
	$$ u'(\bar u)=e^{2\varphi(\bar u)}.$$
	With respect to the new coordinate system we obtain
	$$g=e^{2\varphi(\bar u)}\bar g,$$
	where $$\bar g=e^{-2G}(d\bar x^n)^2+2d\bar vd\bar u+\sum_{i=1}^{n-1}(d\bar x^i)^2+b(\bar u)\sum_{i=1}^{n-1} (\bar x^i)^2 (d\bar u)^2,$$
	$$G(\bar x^n,\bar u)=F(\bar x^n,u(\bar u))+\varphi(\bar u),$$
	and
	$$b(\bar u)=e^{4\varphi(\bar u)}a(u(\bar u))+(\varphi'(\bar u))^2-\varphi''(\bar u).$$
	Requiring the condition $$b(\bar u)=0$$ we obtain the Riccati equation
	$$\varphi''-(\varphi')^2=ae^{4\varphi}.$$
	The function $G(\bar x^n,\bar u)$ satisfies
	$$\ddot G=\dot G^2,$$
	where dot denotes the partial derivative with respect to $\bar u$.
	We obtain the general solution 
	$$G(\bar x^n,\bar u)=-\ln|\bar u+\psi(\bar x^n)|+c(\bar x^n),$$
	where $\psi(x^n)$ and $c(x^n)$ are some functions. 
	
	Changing the notation in the obvious way we obtain the following metric from the local conformal class:
	$$g=e^{-2F}(dx^n)^2+2dvdu+\sum_{i=1}^{n-1}(dx^i)^2,\quad F=-\ln| u+\psi( x^n)|+c( x^n).$$ 
	Note that $$\partial_{x^n}\dot F=\frac{\psi'(x^n)}{(u+\psi(x^n))^2}.$$
	This implies that $\psi'$ is non-vanishing.
	This proves the proposition.
\end{proof}

The Weyl structure obtained in Proposition \ref{lemchangeF} depends on two functions $\psi(x^n)$ and $c(x^n)$. We will see now that after a proper change of the coordinates the Weyl structure depends on a single function.

{\bf Proof of Theorem \ref{ThFinalLocal}.}
By Theorem \ref{MainTh} and Proposition \ref{lemchangeF} there exists a local metric $g$ from the conformal class $c$ given by \eqref{metricag} with $F$ given by \eqref{formulaf}. 
Note that \eqref{eqmain}
implies $$\partial_{x^n}\ddot F-2\dot F\partial_{x^n}\dot F=0.$$
Consequently, $$\partial_u(2F-\ln|\partial_{x^n}\dot F|)=0,$$
i.e., the function $2F-\ln|\partial_{x^n}\dot F|$ depends only on the variable $x^n$.
Consider the new coordinate $t$ related to the coordinate $x^n$ by a function $x^n=x^n(t)$ such that  
\begin{equation}\label{partialxnpartialt}\frac{\partial x^n}{\partial t}=e^{\frac{1}{3}(2F-\ln|\partial_{x^n}\dot F|)}.\end{equation}
Then it holds
$$\partial_t=e^{\frac{1}{3}(2F-\ln|\partial_{x^n}\dot F|)}\partial_{x^n}, \quad dx^n=e^{\frac{1}{3}(2F-\ln|\partial_{x^n}\dot F|)}dt.$$
Consequently, with respect to the coordinates
$t,v,x^1,\dots, x^{n-1},u$ the metric $g$ may be written as
$$g=e^{-2G}(dt)^2+2dvdu+\sum_{i=1}^{n-1}(dx^i)^2,$$
where $$G(t,u)=\frac{1}{3}(F(x^n(t),u)+\ln|\partial_{x^n}\dot F(x^n(t),u)|).$$
From \eqref{formulaf} and \eqref{partialxnpartialt} it follows that the function $c$ satisfies
$$c=\frac{3}{2}\ln\frac{\partial x^n}{\partial t} +\frac{1}{2}\ln|\partial_{x^n}\psi|.$$
Using this and the definition of $G$  we obtain
\begin{multline*}G=\frac{1}{3}\left(-\ln|u+\psi|+c+\ln\frac{|\partial_{x^n}\psi|}{(u+\psi^2)^2}  \right)=-\ln|u+\psi|+\frac{1}{3}(c+\ln|\partial_{x^n}\psi|)\\=-\ln|u+\psi|+\frac{1}{2}\ln\left|\frac{\partial x^n}{\partial t}\partial_{x^n}\psi\right|=-\ln|u+\psi|+\frac{1}{2}\ln|\partial_{t}\psi|.\end{multline*}
Thus $G$ is given by \eqref{funG}.
This implies that $$|\partial_t\dot G|=e^{2G}.$$ 
According to the proof of Proposition \ref{proptildeg},
$$h=e^{\frac{2}{3}(G+\ln|\partial_t\dot G|)}g=e^{2G}g=
(dt)^2+e^{2G}\left(2dvdu+\sum_{i=1}^{n-1}(dx^i)^2\right),$$
and
$$ \omega_h=-\frac{1}{3}\theta=-\frac{1}{3}\partial_{x^n}( F+\ln|\dot F'|) dx^n=-\partial_{x^n}Gdx^n=-\partial_tGdt.$$ 
The theorem is proved. \qed

%
%

Let us consider the recurrent Lorentzian Weyl structures in dimension 3.

\begin{theorem}
	\label{dim3Th}
	Let $(M,c,\nabla)$ be a connected non-closed Lorentzian Weyl manifold of dimension $3$ with recurrent curvature tensor. Then
	one of the following holds:
	\begin{itemize}
		\item[$\bullet$] 		 around each point of $M$ there exist coordinates $v,x,u$ such that $c=[g]$,  $$g = 2dvdu + e^{-2F}(dx)^2,\quad \omega=\dot F du,$$ where
		$F=F(x,u)$ is an arbitrary function with non-vanishing $\partial_x\dot F$; 	
		\item[$\bullet$]
		around each point of $M$ there exist coordinates $v,x,u$ such that $c=[g]$, 
		$$g = 2dvdu + (dx)^2 + H (du)^2,\quad \omega=a(u)xdu,$$ 
		$$H=H(v,x,u) =  a(u)vx+ \frac{1}{12} a^2(u) x^4 -\frac{1}{3} \dot{a}(u) x^3 + c(u)x,$$
		where the function $a(u)$ is non-vanishing.
	\end{itemize}
\end{theorem}

{\bf Proof of Theorem \ref{dim3Th}.}
Let $(M,c,\nabla)$ be a non-closed Lorentzian Weyl manifold of dimension $n + 2 = 3$ with recurrent curvature tensor. Then by~Theorem~\ref{Thhol} its holonomy algebra is one of the following: $$\g = \mathR (\id_{\Real^{1,2}} + p\wedge q),\quad \g = \mathR (2 \id_{\Real^{1,2}} + p\wedge q) \ltimes \mathR p\wedge e_1.$$

Consider the first case. Using the results from~\cite{ParSpin} we obtain that around each point of $M$ there exist coordinates $v,x,u$ and a~metric $g \in c$ such that
$$g = 2dvdu + e^{-2F}(dx)^2 + H (du)^2,\quad \omega= \dot{F}du,$$ 
where $H = H(u)$ and $F = F(x,u)$ are  functions. Applying a simple coordinate transformation, we may assume that $H=0$. 
The only non-zero Christoffel symbols of the connection $\nabla$ are the following:
\begin{align*}
\Gamma^{x}_{xx} = -\partial_x F, \quad
\Gamma^{u}_{uu} = 2 \dot{F}.
\end{align*}
The components of the curvature tensor are as follows
$$R(\partial_v, \partial_x) = 0, \quad
R(\partial_v, \partial_u) = 0, \quad
R(\partial_x, \partial_u) = \partial_x \dot{F} A,\quad A=\begin{pmatrix} 
0 & 0 & 0\\
0 & 1 & 0 \\
0 & 0 & 2 \end{pmatrix}.$$
For  $\nabla R$ we obtain
\begin{align*}
\nabla R(\partial_v, \partial_x) &=\nabla R(\partial_v, \partial_u) = 0, \quad \nabla_v R(\partial_x, \partial_u) = 0\\
\nabla_x R(\partial_x, \partial_u) &= \left( (\partial_x F) (\partial_x \dot{F}) + \partial^2_x \dot{F} \right) A,\\
\nabla_u R(\partial_x, \partial_u) &= \left( \partial_x \ddot{F} - 2 (\partial_x \dot{F}) \dot{F} \right) A.
\end{align*}
The condition $\nabla R=\theta\otimes \Real$ is equivalent to the following system of equations:
\begin{equation}
\theta_v \partial_x \dot{F} = 0, \quad 
\partial^2_x \dot{F} + (\partial_x F) (\partial_x \dot{F}) = \theta_x \partial_x \dot{F}, \quad 
\partial_x \ddot{F} - 2 (\partial_x \dot{F}) \dot{F} = \theta_u \partial_x \dot{F},
\end{equation}
where $\theta_\alpha = \theta (\partial_\alpha)$, $\alpha = v, x, u$.
Let $U_0$ be the set of non-zero points of the function $\partial_x\dot F$. Then on $U_0$ it holds 
$$\theta_v = 0, \quad 
\theta_x =\partial_x(F+\ln|\partial_x\dot F|) ,\quad  
\theta_u = \partial_u( - 2 \dot{F}+\ln|\partial_x\dot F|).$$
Hence, if $U_0\neq\emptyset$, then $U_0$ is the entire coordinate neighbourhood.

Consider the second case. From the results of~\cite{ParSpin} we obtain that around each point of $M$ there exist coordinates $v,x,u$ and a~metric $g \in c$ such that
$$g = 2dvdu + (dx)^2 + H (du)^2,$$ 
where $H$ is a function and the corresponding 1-form $\omega$ satisfies $$\omega=fdu,\quad \partial_v H = f.$$
The Christoffel symbols for the connection $\nabla$ are as follows
\begin{align*}
\Gamma_v = \begin{pmatrix} 
0 & 0 & \frac{1}{2} f \\
0 & 0 & 0 \\
0 & 0 & 0 \end{pmatrix}, \quad
\Gamma_x = \begin{pmatrix} 
0 & -f & \frac{1}{2} \partial_x H \\
0 & 0 & f \\
0 & 0 & 0 \end{pmatrix}, \quad
\Gamma_u = \begin{pmatrix}
\frac{1}{2} f & \frac{1}{2} \partial_x H & -\frac{1}{2} f H + \frac{1}{2} \dot H \\
0 & f & -\frac{1}{2} \partial_x H \\
0 & 0 & \frac{3}{2} f \end{pmatrix}.
\end{align*}

Consider the field of frames 
$$p = \partial_v,\quad e_1 = \partial_x ,\quad q = \partial_u - \frac{1}{2} H \partial_v.$$
It holds $R^{x}_{uvx} = \partial_v f$. On the other hand, from \eqref{algReq1} it follows that  $R(\partial_v, \partial_x) = R(p, e_1) = 0$. Hence, $\partial_v f = 0$. Using this we find the components of the curvature tensor 
\begin{align}
\label{compReq1}
\begin{split}
&R(\partial_v, \partial_x) = 0, \quad
R(\partial_v, \partial_u) = \frac{1}{2} \partial_x f \begin{pmatrix} 
0 & 1 & 0\\
0 & 0 & -1 \\
0 & 0 & 0 \end{pmatrix},\\
&R(\partial_x, \partial_u) = \begin{pmatrix} 
\frac{1}{2} \partial_x f & - \frac{1}{2} f^2 + \dot{f} + \frac{1}{2} \partial_x^2 H & - \frac{1}{2} (\partial_x f) H\\
0 & \partial_x f & \frac{1}{2} f^2 - \dot{f} - \frac{1}{2} \partial_x^2 H \\
0 & 0 & \frac{3}{2} \partial_x f \end{pmatrix}.
\end{split}
\end{align}
From \eqref{algReq1} it follows that 
$$\frac{1}{2} f^2 - \dot{f} - \frac{1}{2} \partial_x^2 H = 0 \quad \text{and} \quad
R(\partial_x, \partial_u) = \frac{1}{2} \partial_x f \begin{pmatrix} 
1 & 0 & -H\\
0 & 2 & 0 \\
0 & 0 & 3 \end{pmatrix}.$$
The covariant derivatives of $R$ are as follows
\begin{align*}
&\nabla_a R(\partial_v, \partial_x) = 0, \quad a = v,x,u,\\
&\nabla_v R(\partial_v, \partial_u) = 0, \quad 
\nabla_x R(\partial_v, \partial_u) = \frac{1}{2} \partial_x^2 f \begin{pmatrix} 
0 & 1 & 0\\
0 & 0 & -1 \\
0 & 0 & 0 \end{pmatrix},\\ 
&\nabla_u R(\partial_v, \partial_u) = \left( \frac{1}{2} \partial_x \dot{f} - \frac{5}{4} f \partial_x f \right) \begin{pmatrix} 
0 & 1 & 0\\
0 & 0 & -1 \\
0 & 0 & 0 \end{pmatrix},\\
&\nabla_v R(\partial_x, \partial_u) = \frac{1}{2} \partial_x^2 f \begin{pmatrix} 
0 & 1 & 0\\
0 & 0 & -1 \\
0 & 0 & 0 \end{pmatrix}, \quad 
\nabla_x R(\partial_x, \partial_u) = \frac{1}{2} \partial_x^2 f \begin{pmatrix} 
1 & 0 & -H\\
0 & 2 & 0 \\
0 & 0 & 3 \end{pmatrix},\\
&\nabla_u R(\partial_x, \partial_u) = \left( \frac{1}{2} \partial_x \dot{f} - \frac{5}{4} f \partial_x f \right) \begin{pmatrix} 
1 & 0 & -H\\
0 & 2 & 0 \\
0 & 0 & 3 \end{pmatrix}.
\end{align*}
The condition $\nabla R=\theta\otimes R$ is equivalent to the following system of equations: $$\partial_x^2 f = 0, \quad \theta_v \partial_x f = 0, \quad \partial_x^2 f = \theta_x \partial_x f, \quad \partial_x \dot f - \frac{5}{2} f \partial_x f = \theta_u \partial_x f,$$
where $\theta_\alpha = \theta (\partial_\alpha)$, $\alpha = v, x, u$.
We see that $f=f(x,u)=a(u)x+b(u)$. Applying a conformal rescaling of the metric and a simple coordinate transformation we may assume that $b(u)=0$.
Let $U_0$ be the set of points where $a(u)$ is non-vanishing. Then on $U_0$, $$\theta_u=\partial_u\ln|a(u)|-\frac{5}{2}f.$$
Since $\theta_u$ is smooth, $U_0$ is the entire coordinate neighborhood.  This proves the theorem. \qed

The curvature tensor of the first Weyl structure from Theorem \ref{dim3Th}  satisfies 
$$\nabla R=\theta  \otimes  R,\quad \theta=\partial_x (F+\ln|\partial_x\dot F|) dx+\partial_u(-2F+ \ln|\partial_x\dot F|)du.$$
Again we obtain $\nabla (R\otimes l^{3}_h)=0$
for $h=e^{\frac{2}{3}\varphi}g$, $\varphi=F+\ln|\partial_x\dot F|$.

For the second structure from Theorem \ref{dim3Th} it holds
$$\nabla R = \theta\otimes  R,\quad \theta= \left(\partial_u\ln|a(u)|- \frac{5}{2} a(u)x \right)du.$$
In this case, $\nabla (R\otimes l^{\frac{5}{2}}_{h})=0$
for $h=e^{\frac{4}{5}\varphi}g$, $\varphi=\ln|a(u)|$.

			\section{The remaining coordinate freedom and symmetries of generic structures}\label{sect:coord}
			
In Section \ref{sec:local}, we used coordinate transformations to simplify the expressions for the recurrent Weyl structures, resulting in Theorem \ref{ThFinalLocal}. Here we investigate the remaining coordinate freedom that preserves these expressions. In doing so, we will also find the infinitesimal symmetries of generic recurrent Weyl structures.

We start with $\dim M \geq 4$. 
\begin{prop} \label{psl2}
Let $(M,h,\omega_h)$ be a non-closed recurrent Lorentzian Weyl manifold of dimension $n+2\geq 4$. In a neighborhood of each point in $M$, there exists local coordinates  so that the pair $(h,\omega_h)$ takes the form 
\begin{equation}
\begin{aligned}
h &= (dt)^2+\frac{\psi'(t)}{(u+\psi(t))^2}\left( 2du dv+\sum_{i=1}^{n-1} (dx^i)^2\right),\\
\omega_h &= \left(\frac{\psi'(t)}{u+\psi(t)}-\frac{\psi''(t)}{2\psi'(t)} \right)dt,
\end{aligned} \label{eq:Weyl}
\end{equation}
with $\psi'(t)>0$. Furthermore, the Lie algebra of vector fields preserving this form is spanned by 
\begin{equation}
\partial_v, \qquad \partial_{x^i}, \qquad 
x^i \partial_v-u \partial_{x^i}, \qquad x^i \partial_{x^j}-x^j \partial_{x^i}, \label{eq:Killing}
\end{equation}
all of which are Killing fields of the metric $h$ that act trivially on $\omega_h$, in addition to the five vector fields
\begin{equation}
\begin{gathered}
\partial_t, \qquad 2 t \partial_t+3 \sum_{i=1}^{n-1} x^i \partial_{x^i}+6v \partial_v, \\
\partial_u, \qquad 2 u \partial_u+\sum_{i=1}^{n-1} x^i \partial_{x^i}, \qquad u^2 \partial_u+u \sum_{i=1}^{n-1} x^i \partial_{x^i} -\frac{1}{2} \sum_{i=1}^{n-1} (x^i)^2 \partial_v.
\end{gathered} \label{eq:aff1psl2}
\end{equation}
\end{prop}


\begin{proof}
According to Theorem \ref{ThFinalLocal} and the paragraph following the theorem, if $(M,c,\nabla)$ is a non-closed recurrent Lorentzian Weyl manifold with $\dim M \geq 4$, then there exists coordinates so that a representative $h\in c$ takes the form
\[ h= (dt)^2+\frac{\psi'(t)}{(u+\psi(t))^2}\left( 2du dv+\sum_{i=1}^{n-1} (dx^i)^2\right),\]
while the corresponding 1-form is given by
\[ \omega_h = \left(\frac{\psi'(t)}{u+\psi(t)}-\frac{\psi''(t)}{2\psi'(t)} \right)dt.\] 
Furthermore, the representative $h \in c$ is unique up to homothety, and $\omega_h$ defines uniquely a 1-dimensional distribution $\langle \partial_t \rangle$. This distribution is parallel: $\nabla \partial_t \in \langle \partial_t \rangle$.

If $\varphi_s$ is a 1-parameter group of transformations preserving the coordinate form of the recurrent Weyl structure given by $(h,\omega_h)$, then we have 
\begin{align*}
  \varphi_s^* h &= e^{2\lambda_s} \left((dt)^2+\frac{\mu_s'(t)}{(u+\mu_s(t))^2}\left( 2du dv+\sum_{i=1}^{n-1} (dx^i)^2\right) \right), \\
  \varphi_s^* \omega_h &= \left(\frac{\mu_s'(t)}{u+\mu_s(t)}-\frac{\mu_s''(t)}{2\mu_s'(t)} \right)dt-d\lambda_s. 
\end{align*}
Since the representative $h$ of the conformal class is unique up to homothety, it is clear that $\lambda_s$ is constant, while $\mu_s$ is a function of $t$ only. The equations for the vector field $Y$ whose flow is $\varphi_s$ are then given by 
\begin{equation}
\begin{split}
L_Y h &= 2 \tilde\lambda h + \left(\frac{\tilde \mu'(t)(u+\psi(t))-2\psi'(t) \tilde \mu (t)}{(u+\psi(t))^3}\right)\left( 2du dv+\sum_{i=1}^{n-1} (dx^i)^2\right), \\
L_Y \omega_h &= \left(\frac{\tilde \mu'(t)(u+\psi(t))-\psi'(t)\tilde \mu(t)}{(u+\psi(t))^2} - \frac{\tilde \mu''(t) \psi'(t)-\psi''(t)\tilde \mu'(t)}{2 \psi'(t)^2} \right) dt-d\tilde \lambda,
\end{split} \label{eq:Ysys}
\end{equation}
where $\tilde \lambda = \frac{d}{ds}\big|_{s=0} \lambda_s$ is a constant and $\tilde \mu = \frac{d}{ds}\big|_{s=0} \mu_s$ is a function of $t$. The right-hand-side of these equations follows from the definition of the Lie derivative: $L_Y h= \frac{d}{ds} \big|_{s=0} (\varphi_s^* h)$ and $L_Y \omega_h= \frac{d}{ds} \big|_{s=0} (\varphi_s^* \omega_h)$. The vector fields \eqref{eq:Killing} and \eqref{eq:aff1psl2} can be found by a direct analysis of \eqref{eq:Ysys}, but the computations are easier to follow if we focus on the parallel vector fields and distributions admitted by \eqref{eq:Weyl}. 

The general recurrent Weyl structure \eqref{eq:Weyl} admits a unique parallel (null) vector field $\partial_v$. It also admits the parallel distributions $\langle \partial_t \rangle$,  $\langle \partial_t \rangle^\perp = \langle \partial_u,\partial_v,\partial_{x^1},\dots, \partial_{x^{n-1}} \rangle$, $\langle \partial_v \rangle^\perp \cap \langle \partial_t \rangle^\perp = \langle  \partial_v,\partial_{x^1},\dots, \partial_{x^{n-1}} \rangle$, and it is clear that $Y$ must preserve these distributions. 

Let us write $Y$ in coordinates:
\[Y = A \partial_v+\sum_{i=1}^{n-1} B^i \partial_{x^i}+B^n \partial_t+C \partial_u,  \]
where $A,B^1,\dots,B^{n}, C$ are smooth functions. They must satisfy the following conditions:
\begin{align*}
  L_Y \partial_v = -A_0 \partial_v \qquad & \Leftrightarrow \qquad \partial_v A=A_0, \partial_v B^1=\cdots = \partial_v B^n = \partial_v C =0, \\
  L_Y \partial_t \in \langle \partial_t \rangle  \qquad & \Leftrightarrow  \qquad  \partial_t A=\partial_t B^1=\cdots = \partial_t B^{n-1} = \partial_t C =0, \\
  L_Y \partial_{x^i} \in \langle \partial_v, \partial_{x^i}, \dots, \partial_{x^{n-1}} \rangle \qquad & \Leftrightarrow \qquad \partial_{x^i} C = \partial_{x^i} B^n = 0, \\
  L_Y \partial_u \in \langle \partial_t \rangle^\perp \qquad & \Leftrightarrow \qquad \partial_u B^n =0, \\
  (L_Y h)(\partial_u,\partial_u) =0 \qquad &\Leftrightarrow \qquad \partial_u A =0, \\
  (L_Y h)(\partial_{x^i},\partial_{x^j}) = 0 \qquad &\Leftrightarrow \qquad \partial_{x^i} B^j + \partial_{x^j} B^i =0, \\
   (L_Y h)(\partial_{u},\partial_{x^i}) = 0 \qquad &\Leftrightarrow \qquad \partial_{x^i} A + \partial_{u} B^i =0, \\
   (L_Y h)(\partial_{x^i}, \partial_{x^i}) =  (L_Y h)(\partial_{x^j}, \partial_{x^j})  \qquad &\Leftrightarrow \qquad \partial_{x^i} B^i = \partial_{x^j} B^j,\\
   (L_Y h)(\partial_{u},\partial_{v}) = (L_Y h)(\partial_{x^i}, \partial_{x^i}) \qquad &\Leftrightarrow \qquad \partial_v A+ \partial_u C=2 \partial_{x^i} B^i, \\
  (L_Y h)(\partial_t,\partial_t)=2 \tilde \lambda \qquad & \Leftrightarrow \qquad \partial_t B^n=\tilde \lambda.
\end{align*}
Here $A_0$ and $\tilde \lambda$ are constants. The general solution of this system of differential equations is parametrized by $(2n-1)+\binom{n-1}{2}+6$ constants, including $\tilde \lambda$ and $A_0$. Assuming that $Y$ satisfies these conditions, the requirement that a function $\tilde \mu(t)$ satisfying the first line of \eqref{eq:Ysys} exists  gives one additional relation between the constants, namely $A_0=3 \tilde \lambda$ (this is independent of $\psi$).  It reduces the number of parameters to  $(2n-1)+\binom{n-1}{2}+5$, and the  vector fields satisfying all of these conditions are exactly those listed in the proposition. They automatically satisfy the second line of \eqref{eq:Ysys}. 

The result was checked with the help of Maple's DifferentialGeometry package. It is straight-forward to verify that \eqref{eq:Killing} are Killing fields of the metric $h$ for any $\psi$. 
\end{proof}

Proposition \ref{psl2} says that all  non-closed  recurrent Lorentzian Weyl structures on manifolds of  dimension greater than 3 admit the infinitesimal symmetries \eqref{eq:Killing}, in some local coordinate system. These $(2n-1) + \binom{n-1}{2}$ vector fields span a Lie algebra abstractly isomorphic to $\mathfrak{heis}(n-1) \rtimes \mathfrak{so}(n-1)$. The Lie algebra spanned by \eqref{eq:aff1psl2} is abstractly $\mathfrak{aff}_1(\mathbb R) \oplus \mathfrak{sl}_3(\mathbb R)$. The full Lie algebra spanned by the vector fields of \eqref{eq:Killing} and \eqref{eq:aff1psl2} has Levi decomposition 
\[(\mathfrak{so}(n-1) \oplus \mathfrak{sl}_2(\mathbb R)) \ltimes (\mathfrak{aff}_1(\mathbb R) \ltimes \mathfrak{heis}(n-1)).\]
Any additional infinitesimal symmetry is a linear combinations of the vector fields \eqref{eq:aff1psl2}. To justify this last sentence, we note that the  condition for a vector field $Y$ to be a symmetry of the Weyl structure defined by $\psi$ is given by equations \eqref{eq:Ysys} with $\tilde \mu \equiv 0$, and it is not difficult to see that a symmetry is required to satisfy all of the conditions that a vector field preserving the form \eqref{eq:Weyl} satisfies.
For a generic  non-closed recurrent Lorentzian Weyl structure the flow of any linear combination of the vector fields \eqref{eq:aff1psl2} acts nontrivially. Thus a generic non-closed recurrent Lorentzian Weyl structure admits exactly $(2n-1) + \binom{n-1}{2}$  infinitesimal symmetries.

Integrating the first two vector fields of \eqref{eq:aff1psl2} results in the transformations 
\[u \mapsto u, \quad  x^i \mapsto e^{3s_2} x^i, \quad  t \mapsto e^{2s_2}t+s_1, \quad  v \mapsto  e^{6s_2} v, \qquad s_1,s_2 \in \mathbb R, \]
while integrating the last three vector fields of \eqref{eq:aff1psl2} results in the transformations
\[
 u \mapsto \frac{au+b}{cu+d}, \;\;  x^i \mapsto \frac{x^i}{cu+d}, \;\;  t \mapsto t, \;\;  v \mapsto  v+\frac{c/2}{cu+d} \sum_{i=1}^{n-1} (x^i)^2, \quad a,b,c,d \in \mathbb R,\; ad-bc=1.
\]
Under these transformations, the function $\psi$ is  transformed according to 
\[\psi(t) \mapsto \psi(e^{-2 s_2}(t-s_1)), \qquad  \psi(t) \mapsto \frac{a\psi(t)-b}{-c\psi(t)+d}, \]
respectively. Notice that the transformations parametrized by $a,b,c,d$ commute with the ones that are parametrized by $s_1,s_2$. The structure of this transformation group is $\SAff \times \PSL$. This product is the connected Lie group whose action gives rise to the Lie algebra of vector fields spanned by \eqref{eq:aff1psl2} ($\SAff$ denotes the connected component of the group $\Aff$ of affine transformations on $\mathbb R$). This group can be extended with the transformation 
\[  u\mapsto -u, \quad  x^i \mapsto x^i, \quad  t \mapsto-t, \quad  v \mapsto- v\]
which also preserves the form \eqref{eq:Weyl} and the condition $\psi'(t)>0$. Under this transformation $\psi$ is transformed as $\psi(t)\mapsto - \psi(-t)$.

Notice that the two transformations $(u,x^i,t,v) \mapsto (u,x^i,-t,v)$ and $(u,x^i,t,v) \mapsto (-u,x^i,t,-v)$ of the extended groups $\Aff$ and $\PGL$, respectively, also preserve the coordinate form of Theorem \ref{ThFinalLocal}, but not the condition $\psi'(t)>0$. 

To sum up, we have the following transformations on $\psi$,  along with combinations of such: 
\begin{equation}
\begin{aligned}
\psi(t) &\mapsto \psi(e^{-2s_2} (t-s_1)), \\
\psi(t) &\mapsto \frac{a \psi(t)-b}{-c \psi(t) + d} , \quad ad-bc =1, \\
\psi(t) &\mapsto  -\psi(-t).
\end{aligned} \label{eq:groupactionD4}
\end{equation}
This can be viewed as an action of the group $\SAff \times \PSL \times \mathbb Z_2 \subset \Aff \times \PGL$ on the space $\{\psi \in C^\infty(\mathbb R) \mid \psi'(t)>0 \}$. 

\begin{cor}
The local equivalence problem of  non-closed recurrent Lorentzian Weyl manifolds of dimension $\geq 4$ can be reduced to the local equivalence problem of functions $\{\psi \in C^\infty(\mathbb R) \mid \psi'(t) > 0\}$ under the Lie group action \eqref{eq:groupactionD4}.
\end{cor}

Next, we set $\dim M=3$ and start by considering the first case from Theorem \ref{dim3Th} where the recurrent Weyl structure  $(g,\omega_g)$ is given by
\begin{equation}
 g = 2 dv du + e^{-2F(x,u)} (dx)^2, \qquad \omega_g = \dot F du \label{eq:3D1}
\end{equation}
with $\partial_x \dot F$  non-vanishing. 

\begin{prop}\label{prop:3D1}
A vector field preserves the coordinate form of the general recurrent Weyl structure \eqref{eq:3D1} if and only if it takes the form 
\[ A(x) \partial_x+B(u) \partial_u+(C_0 + C_1 v) \partial_v\]
where $A,B$ are smooth functions and $C_0,C_1$ are constants. Moreover, the vector field $\partial_v$ is an infinitesimal symmetry for every recurrent Weyl structure of the form \eqref{eq:3D1}, and a generic recurrent Weyl structure of this form admits no additional infinitesimal symmetries.
\end{prop}
\begin{proof}
Each structure of the form \eqref{eq:3D1} admits the unique (up to a constant factor) $\nabla$-parallel null vector field $\partial_v$. Any infinitesimal symmetry $Y$ is required to preserve this vector field: $L_Y \partial_v=-C_1 \partial_v$ for some constant $C_1$. This gives
\[ Y= A(x,u) \partial_x+B(x,u) \partial_u+(C_0(x,u)+C_1 v) \partial_v.\] 
The condition $(L_Y g)(\partial_u,\partial_u)=0$ implies that $ \dot C_0(x,u)=0$, while $(L_Y g)(\partial_x, \partial_v)=0$ implies that $\partial_x B(x,u)=0$. The condition for $Y$ to preserve the form \eqref{eq:3D1} can be written as
\begin{align*}
 L_Y g &= 2  \lambda g -2 \mu e^{-2F(x,u)} (dx)^2, \\
 L_Y \omega_g &=\dot \mu du - d \lambda.
\end{align*} 
For $Y$ with $B(x,u)=B(u)$ and $C_0(x,u)=C_0(x)$ these equations are equivalent to
\[\lambda=\frac{C_1+\dot B(u)}{2}, \qquad \mu = B(u)  \dot F(x,u)+A(x,u) \partial_x F(x,u)+\lambda-\partial_x A(x,u),\]
 together with $f_1:=C_0'(x)+\dot A(x,u) e^{-2F(x,u)} =0$ and $f_2:=\dot A(x,u) \partial_x F(x,u)-\partial_x \dot A(x,u)=0$. We have 
 \[2 (\partial_x \dot f_1+\dot f_1 \partial_x F) e^{2F}+2\dot f_2-4 f_2 \dot F= - 2 \dot A \partial_x \dot F, \]
 which implies $\dot A=0$. Setting $A(x,u)=A(x)$ implies that $f_2=0$, and the remaining condition, $f_1=0$, is equivalent to $C_0(x)$ being constant, and we arrive at the general vector field in the proposition. It is straight-forward to show that $\partial_v$ is a Killing vector for all Weyl structures of the form \eqref{eq:3D1}, and to verify that generic recurrent Weyl structures of the form \eqref{eq:3D1} admit no additional infinitesimal symmetries. The latter will also be clear from the computations in Section \ref{sect:symmetric}. 
\end{proof}

The Lie pseudogroup preserving the coordinate form of \eqref{eq:3D1} is given by 
\[ (v,x,u) \mapsto (C_1 v+C_0,\alpha(x), \beta(u)), \qquad \alpha'(x) \neq 0, \quad \beta'(u) \neq 0, \quad C_1 \beta'(u) > 0.\]
Here $\alpha, \beta$ are local diffeomorphisms on the line, and $C_0,C_1$ are constants. The connected component of the Lie pseudogroup is constrained by $\alpha'(x)>0, \beta'(u) >0, C>0$ (for each $(x,u)$ where the transformations are defined). However, the extension given by $\alpha'(x) \neq 0,  \beta'(u) \neq 0,  C \beta'(u) >0$  also preserves the coordinate form of the Weyl structure, giving us a Lie pseudogroup consisting of 4 connected components. 

We get the following Lie pseudogroup action on the the space $C^\infty(\mathbb R^2)$:
\begin{equation}
\begin{split}
F(x,u) &\mapsto F\left(\alpha^{-1}(x),\beta^{-1}(u)\right)-\frac{1}{2} \ln\left(C_1\frac{ (\alpha^{-1})'(x)^2}{(\beta^{-1})'(u)}\right), \\
\alpha'(x) &\neq 0, \qquad \beta'(u) \neq 0, \qquad C_1 \beta'(u) >0.
\end{split} \label{eq:groupaction3D1}
\end{equation}


Next we consider the recurrent Weyl structures from Theorem \ref{dim3Th} that locally take the form 
\begin{equation}
\begin{aligned}
g = 2dvdu + (dx)^2 + H (du)^2,\qquad \omega_g=a(u)xdu, \\
H=  a(u)vx+ \frac{1}{12} a^2(u) x^4 -\frac{1}{3} \dot{a}(u) x^3 + c(u)x,
\end{aligned} \label{eq:3D2}
\end{equation}
where the function $a(u)$ is non-vanishing. 

\begin{prop}\label{prop:3D2}
The Lie algebra of vector fields preserving the coordinate form of the general recurrent Weyl structure \eqref{eq:3D2} is spanned by the four vector fields
\begin{equation*}
  \partial_u, \qquad \partial_v, \qquad u \partial_u-v\partial_v, \qquad u \partial_u+v\partial_v+x \partial_x.
\end{equation*}
A generic Weyl structure of this form admits no infinitesmial symmetries. 
\end{prop}
\begin{proof}
Taking into account that an infinitesimal symmetry must preserve the unique parallel distribution spanned by $\partial_v$, we see that an infinitesimal symmetry must be of the form 
\[ Y = A(v,x,u)\partial_v+B(x,u) \partial_x+C(x,u)\partial_u.\] 
The functions
\[ (L_Y g)(\partial_v,\partial_x), \quad (L_Y g)(\partial_x,\partial_u)\] 
vanish if and only if $\partial_x C=0$ and $\partial_x A+\partial_u B=0$.  Furthermore, the conditions 
\[(L_Y g)(\partial_x,\partial_x) = 2B_0 = (L_Y g)(\partial_u,\partial_v)\] are equivalent to $\partial_x B=B_0$ and $\partial_v A+\partial_u C=2B_0$. Solving this overdetermined system of 4 PDEs gives 
\[ A=F_3(u)+(2B_0-F_2'(u))v-F_1'(u)x, \quad B=B_0 x+F_1(u),\quad C=F_2(u).\]
The condition $L_Y \omega_g = \tilde a(u) x du$ implies $F_1(u) \equiv 0$. In order for the equations 
\begin{align*}
  L_Y g &= 2 \tilde \lambda g + \left(\tilde a(u) vx+\frac{2}{12}  a(u) \tilde a(u) x^4-\frac{1}{3} \tilde a'(u) x^3 + \tilde c(u) x\right) (du)^2, \\
  L_Y \omega &= \tilde a(u) x du - d\tilde \lambda
\end{align*} 
to hold, for some pair of functions $\tilde a(u), \tilde c(u)$, we get the additional requirements that $F_3$ is constant, $F_2$ is affine in $u$ and $F_1$ vanishes identically. This means that $Y$ is a linear combination of the four vector fields from the proposition.   It is easy to verify that generic recurrent Weyl structures of the form \eqref{eq:3D2} admit no infinitesimal symmetries. This will also be clear from the computations in Section \ref{sect:symmetric}. 
\end{proof}

The Lie group preserving the coordinate form of \eqref{eq:3D2} is given by
\[ (v,x,u) \mapsto \left(\frac{A_4}{A_3} v+A_2, A_4 x, A_3 A_4 u +A_1 \right), \qquad A_1,A_2 \in \mathbb R, \quad A_3,A_4 \in \mathbb R\setminus \{0\}.\]
Thus we are left with a 4-dimensional Lie subgroup of $\mathrm{Aff}(3)$ generated by 
\begin{equation}
\begin{split}
 (a(u),c(u)) \mapsto &\left(\frac{a\left(\frac{u-A_1}{A_3 A_4}\right)}{A_3 A_4^2},  \frac{c\left(\frac{u-A_1}{A_3 A_4}\right)}{A_3^{2} A_4}-\frac{A_2a\left(\frac{u-A_1}{A_3 A_4}\right)}{ A_3 A_4^{2} } \right), \\ &\;\; A_1,A_2 \in \mathbb R, \quad A_3,A_4 \in \mathbb R\setminus \{0\},
 \end{split} \label{eq:groupaction3D2}
 \end{equation}
 acting on the pair $(a(u),c(u))$ of functions. This Lie group consists of 4 connected components. 

\begin{cor}
(1) The local equivalence problem of  non-closed recurrent Lorentzian Weyl manifolds of dimension $3$ with 1-dimensional holonomy algebra can be reduced to the local equivalence problem of functions in $C^\infty(\mathbb R^2)$ under the Lie pseudogroup action \eqref{eq:groupaction3D1}. 
(2) The local equivalence problem of  non-closed recurrent Lorentzian Weyl manifolds of dimension $3$ with 2-dimensional holonomy algebra can be reduced to the local equivalence problem of a pair of functions  in $C^\infty(\mathbb R)^2$ under the Lie group action \eqref{eq:groupaction3D2}. 
\end{cor}

\section{The global form of recurrent Weyl structures} \label{sec:global}

Let $(M,c,\nabla)$ be a non-closed recurrent Lorentzian Weyl  manifold and $\dim M=n+2\geq 4$. In this section we assume that $M$ is simply connected.


{\bf Proof of Proposition \ref{Propmetrh}.} Let us choose an open cover $\{U_\alpha\}$ of $M$ such that for each $U_\alpha$ there exists a metric $h_\alpha$ as in  Proposition \ref{proptildeg}. On the intersection $U_\alpha\cap U_\beta$ it holds that
$$h_\alpha=e^{2c_{\alpha\beta}}h_\beta$$ for some $c_{\alpha\beta}\in\Real$. Clearly this defines a \v{C}ech 1-cocycle. Since the manifold is simply connected, there exist numbers $b_\alpha$ such that $b_\alpha-b_\beta=c_{\alpha\beta}$. The local metrics $e^{-2b_\alpha}h_\alpha$ define a global metric $h$ on $M$.

Let us show that $h\in c$. We consider again the cover $\{U_\alpha\}$. Let now $h_\alpha$ be the restriction of $h$ to $U_\alpha$. In the proof of Theorem \ref{MainTh} we used the special local coordinates constructed in \cite{ParSpin}. This construction implies that for each $h_\alpha$ there exists $g_\alpha\in c$ and a function $f_\alpha$ on $U_\alpha$ such that $$h_\alpha=e^{2f_\alpha}g_\alpha|_{U\alpha}.$$
Let us fix an index $\gamma$. Since for any $\alpha$ we have $g_\alpha,g_\gamma\in c$, there exists a function $\varphi_{\alpha\gamma}$ on $M$ such that
$$g_\alpha=e^{2\varphi_{\alpha\gamma}}g_\gamma.$$
Since on the intersection $U_\alpha\cap U_\beta$ it holds  $h_\alpha=h_\beta$, we conclude that  $$f_\alpha+\varphi_{\alpha\gamma}=f_\beta+\varphi_{\beta\gamma}$$ on $U_\alpha\cap U_\beta$,
i.e., the functions $f_\alpha+\varphi_{\alpha\gamma}$ define a global function  $f$ on $M$ and 
$$h=e^{2f}g_\gamma\in c.$$
\qed

{\bf Proof of Theorem \ref{Thglobal}.}
Fix a metric $h\in c$ as in Proposition \ref{Propmetrh}. It is clear that $h$ is defined uniquely up to homothety.
Fix a vector field $X$ as in the Introduction.
Consider a local Weyl structure as in  Theorem \ref{ThFinalLocal} defined on an open subset of $M$. The vector field $\partial_t$
satisfies $ h(\partial_t,\partial_t)=1$. This shows that locally $X=\pm \partial_t$.
The coordinate $t$ may be chosen in such a way that $X= \partial_t$.
Consider the dual 1-from 	
$\eta$ on $M$ defined by $\eta(Y)= h (X,Y)$. Locally $\eta=dt$, consequently,  $$d\eta=0$$ on $M$.
Next, $$\eta(X)=h(X,X)=1.$$
Recall that by the assumption, the vector field $X$ is complete, and the manifold $M$ is simply connected. By \cite[Proposition 8]{LSch},
$M$ is globally diffeomorphic to the product
$$M\cong \Real\times N,$$
where $N$ is a leaf of the foliation defined by the kernel of the 1-from $\eta$. Denote by $t$ the coordinate on $\Real$. Under the diffeomorphism, the vector field $X$ is identified with $\partial_t$.
It is obvious that the just constructed coordinate $t$ is the globalization of the local coordinate $t$ considered in Theorem \ref{ThFinalLocal}. 
Consequently, in a neighborhood $U$ of each point from $N$ there exist    coordinates $v,x^1,\dots,x^{n-1},u$ such that on $\Real\times U$ the metric $h$ is as in Theorem \ref{ThFinalLocal} with a global function  $\psi$ on $\Real$. 
Recall that the holonomy algebra of the connection $\nabla$ annihilates an isotropic vector. Since $M$ is simply connected,  there exists a $\nabla$-parallel isotropic vector field $p$ on $M$. Locally $p$ is proportional to $\partial_v$, and we may assume that all coordinate systems are choosen in such a way that locally it holds $\partial_v=p$.
The orthogonal distribution on $N$ to $p$ defines a foliation ${\mathcal F}$ on $N$. It is clear that the coordinates $v,x^1,\dots,x^{n-1},u$ are foliated with $u$ being the transversal coordinate.
The rest of the proof follows immediately from the results of Section \ref{sect:coord}. \qed

\section{Differential invariants and the local equivalence problem}\label{sect:diffinv}
By Theorems \ref{ThFinalLocal} and \ref{dim3Th}, any non-closed recurrent Lorentzian Weyl structure on a manifold $M$ is locally defined by a function $\psi(t)$ if $\dim M\geq 4$, or by either one function $F(u,x)$ or a pair $(a(u),c(u))$ of functions when $\dim M=3$. In Section \ref{sect:coord} we found the Lie pseudogroups representing the remaining coordinate transformations acting on these structures. In this section we compute the differential invariants of these Lie pseudogroups. 

Let us start with briefly outlining the approach. For more details on differential invariants and the theory of jet bundles, we refer to \cite{O,Handbook,KL} and references therein. Consider a bundle $\pi \colon E \to B$, and let $G$ be a Lie pseudogroup of local diffeomorphisms on $E$ preserving the fibers of $\pi$. For each $p \in B$ and $k \geq 0$, let $J_p^k \pi$ denote the space of $k$-jets of sections of $\pi$ at the point $p$, and define the $k$-th order jet space $J^k \pi = \sqcup_{p \in B} J_p^k \pi$. The space $J^k \pi$ is a bundle over $J^l \pi$ for $0 \leq l<k$, and over $B$. We have an identification $J^0 \pi \simeq E$. A section $s$ of $\pi$ is naturally prolonged to a section $j^k s$ of $J^k \pi$.

The Lie pseudogroup $G$ prolongs uniquely to $J^k \pi$ for any $k \geq 1$, determined by how it transforms sections of $\pi$. There exists formulas for these prolongations (see for example (1) of  \cite{Handbook} to see how a vector field on $E$ is prolonged to $J^k \pi$), but this type of computations is very tedious and is best to leave to a computer algebra system. For the computations done here, Maple with its DifferentialGeometry and JetCalculus packages was used extensively. 

A differential invariant of order $k$ is a function on $J^k \pi$ which is constant on $G$-orbits. If $I$ is a differential invariant of order $k$ and $s$ is a section of $\pi$, we get a function $I \circ j^k s$ on $B$. The examples we work with here satisfy the conditions of \cite{KL} and it is therefore sufficient to consider differential invariants that are rational functions on the fibers of $J^k \pi \to J^0 \pi$, as they will separate generic orbits in $J^k \pi$. After finding generators for the differential invariants in each of our cases, we will show how to use them to solve the equivalence problem. 

Note that it is possible to compute differential invariants directly from the local form of the recurrent Weyl structures, but then one works on a bundle with 3-dimensional base, rather than a bundle with 1- or 2-dimensional base  which we are doing here. Furthermore, the sections of this bundle would be constrained by a nontrivial system of PDEs. The choice we have made here simplifies the exposition but gives the same end result.

The space of functions on $\mathbb R$ can be identified with the space of sections of the trivial bundle $\pi \colon \mathbb R \times \mathbb R \to \mathbb R$. Let $t,\psi_0,\psi_1,\dots, \psi_k$ be coordinates on the space $J^k \pi$, and let $D_t$ denote the total derivative operator:
\[ D_t = \partial_t + \sum_{i=0}^\infty \psi_{i+1} \partial_{\psi_i}.\] 

\begin{prop} \label{Prop:invariants}
 The field of rational differential invariants of functions on $\mathbb R$  under the Lie group action \eqref{eq:groupactionD4} are generated by the two differential invariants
 \[ I= \frac{(\psi_1^2 \psi_4-4 \psi_1 \psi_2 \psi_3+3 \psi_2^3)^2}{(2 \psi_1 \psi_3-3 \psi_2^2)^3}, \qquad J= \frac{\psi_1 (\psi_1^2 \psi_5-5 \psi_1 \psi_2 \psi_4+5 \psi_2^2 \psi_3)}{(2 \psi_1 \psi_3-3 \psi_2^2)^2},\]
 and the invariant derivation $\hat \partial_I = D_t(I)^{-1} D_t$.
\end{prop}  
\begin{proof}
The dimension of a generic orbit on $J^k \pi$ is equal to $5=\dim \mathfrak g$ for $k \geq 3$. Since $\dim J^5 \pi =7$ there exist two algebraically independent invariants on $J^5(\mathbb R)$. It is not difficult to verify that $I$ and $J$ are invariants on $J^5 \pi$ and that they are independent. For each $k \geq 5$, we can construct a new independent invariant on $J^k(\mathbb R)$ by applying the invariant derivation $\hat \partial_I$
$k-5$ times to $J$. Thus we have a transcendence basis of differential invariants of any fixed order $k$. Since all the invariants, except for $I$, are affine in the highest-order variables, it is not difficult to see that they generate the whole field, i.e., that there exists no algebraic extension consisting of rational invariants only.  
\end{proof}


For the next proposition, we work over the bundle $\pi \colon \mathbb R \times \mathbb R^2 \to \mathbb R$, and use coordinates $u$, $a_0$, $c_0$, $a_1$, $c_1$, $\dots$, $a_k$, $c_{k}$ on $J^k \pi$. Let $D_u$ denote the total derivative operator: 
\[ D_u = \partial_u + \sum_{i=0}^\infty (a_{i+1} \partial_{a_i}+b_{i+1} \partial_{b_i}).\] 

\begin{prop} \label{Prop:invariants3D2}
The field of rational differential invariants of pairs of functions on $\mathbb R$ under the group action \eqref{eq:groupaction3D2} is generated by the differential invariants
\begin{gather*}
I = \frac{(a_0 c_1-c_0 a_1) a^4}{a_1^4}, \quad J=\frac{a_0 a_{2}}{a_1^2}, \quad K=\frac{(a_0 c_{2}-c_0 a_{2}) a^5}{a_1^5},
\end{gather*}
and the invariant derivation $\hat \partial_I=D_u(I)^{-1} D_u$.  
\end{prop} 
\begin{proof}
Generic orbits on $J^k \pi$ are 4-dimensional for $k\geq 1$. It can easily be verified that $I, J, K$ and $\hat \partial_I$ are invariant, and that they generate $2k-1=\dim J^k \pi - 4$ algebraically independent invariants on $J^k$. These $2k-1$ algebraically independent invariants can be taken to be affine in the highest-order variables, implying that they generate the whole field of rational invariants on $J^k \pi$ for any $k$. 
\end{proof}

In the next case we work over the bundle $\pi \colon \mathbb R^2 \times \mathbb R \to \mathbb R^2$ and use the coordinates $u$, $x$, $F$, $F_u$, $F_x$, $\cdots$, $F_{u x^{k-1}}$,  $F_{x^k}$ on $J^k \pi$.  Let $D_u$ and $D_x$ denote the total derivative operators (here $F_{u^0 x^0}=F$): 
\[ D_u = \partial_u + \sum_{k=0}^\infty \sum_{i=0}^k F_{u^{k-i+1} x^i} \partial_{F_{u^{k-i} x^i}}, \quad D_x = \partial_x + \sum_{k=0}^\infty \sum_{i=0}^k F_{u^{k-i} x^{i+1}} \partial_{F_{u^{k-i} x^i}}.\] 

\begin{prop}\label{Prop:invariants3D1}
The field of rational differential invariants of functions on $\mathbb R^2$ under the Lie pseudogroup action \eqref{eq:groupaction3D1} is generated by the differential invariants 
\begin{align*}
I &= \frac{(-2  F_u  F_{ux}+ F_{uux}) ( F_x  F_{ux}+ F_{uxx})}{ F_{ux}^3}, \\
J &= \frac{(-2  F_u  F_x  F_{ux}-2  F_u  F_{uxx}+ F_x  F_{uux}+ F_{uuxx})}{ F_{ux}^2},
\end{align*}
and the invariant derivations 
\[ \nabla_1 = \frac{F_{uxx}+F_x F_{ux}}{F_{ux}^2} D_u, \qquad   \nabla_2 = \frac{F_{uux}-2F_u F_{ux}}{F_{ux}^2} D_x.\]
\end{prop}

\begin{proof}
The Lie pseudogroup acts transitively on $J^2 \pi$, while the codimension of a generic orbit on $J^k$ is $\binom{k}{2}-2$ for $k \geq 3$. It is easy to verify that $I,J, \nabla_1,\nabla_2$ are invariant, and that the 4 invariants $I,J,\nabla_1(I),\nabla_2(I)$ are algebraically independent on $J^4 \pi$. It is possible to make algebraic combinations of the 4th-order invariants to obtain 3 independent invariants $K_1, K_2, K_3$ in such a way that $K_1$ depends only on $F_{uuux}$ and lower-order variables, $K_2$ depends only on $F_{uuxx}$ and lower-order variables and $K_3$ depends only on $F_{uxxx}$ and lower-order variables. Due to the special form of $\nabla_1$ and $\nabla_2$ it is now obvious that we can generate 4 new independent invariants of order 5, 5 new independent invariants of order 6, and so on. This gives us in total $\binom{k}{2}-2$ invariants on $J^k \pi$. All invariants above are (or can be chosen to be) affine in the highest-order variables, implying that they generate the whole field of rational invariants on $J^k \pi$ for any $k$. 
\end{proof}

We finish this section with explaining how the differential invariants can be used for solving the local equivalence problem of generic  non-closed  recurrent Lorentzian Weyl structures. The idea is essentially that of ``classifying manifolds'' in \cite{O}. Let us focus on the case covered by Proposition \ref{Prop:invariants3D1}. This is the least trivial case, since the base of $\pi$ is 2-dimensional. 

First, let us note that there are other ways of generating the field of differential invariants.  If we have two independent differential invariants $I_1$ and $I_2$ of order $k$, we can construct invariant derivations in the following way. Take $\hat \partial_{I_i} = a_i D_u + b_i D_x$ for $i=1,2$ where $a_i,b_i$ are undetermined functions on $J^k \pi$ for some $k$. The condition $\hat \partial_{I_i}(I_j)= \delta_{ij}$ determines $a_1,b_1,a_2,b_2$ uniquely. The invariant derivations $\hat \partial_{I_1},\hat \partial_{I_2}$ are called the Tresse derivatives corresponding to the pair $I_1,I_2$. In general, the Tresse derivatives can not be constructed from every pair of algebraically independent invariants. We must require that the linear system above has a unique solution. If it does, it follows that the 2-form $d(I_1 \circ j^k s) \wedge d(I_2 \circ j^k s)$ is nonzero for a generic section $s \in \Gamma(\pi)$, i.e., the functions $I_1 \circ j^k s$ and $I_2 \circ j^k s$ are functionally independent functions on $B$. 

Now, consider the two invariants $I,J$ of Proposition \ref{Prop:invariants3D1}, and define $K=\nabla_1(I),L=\nabla_2(I)$. Using the Tresse derivatives corresponding to $I,J$ from Proposition \ref{Prop:invariants3D1}, we can generate all invariants by applying $\hat \partial_I$ and $\hat \partial_J$ to $K$ and $L$. For a section $s \in \Gamma(\pi)$ we can use the four invariants to define four functions on $B$: 
\[I_s= I \circ j^3 s, \quad J_s=J \circ j^4 s, \quad K_s=K \circ j^4 s, \quad L_s=L \circ j^4 s.\]
If $s$ is sufficiently generic, the functions $I_s$ and $J_s$ are functionally independent, and can be used as local coordinates. The remaining two functions $K_s$ and $L_s$ can be expressed in terms of them: $K_s=\kappa_s(I_s,J_s), L_s = \lambda_s(I_s, J_s)$. Two generic sections $s_1$ and $s_2$ are locally equivalent if and only if $(\kappa_{s_1}, \lambda_{s_1})$ is equal to $(\kappa_{s_2}, \lambda_{s_2})$ in some neighborhood. 

For Proposition \ref{Prop:invariants} the local equivalence problem is solved for generic sections $s_1$ and $s_2$ by comparing the function $J_{s_1} = \kappa_{s_1}(I_{s_1})$ to $J_{s_2} = \kappa_{s_2}(I_{s_2})$ in a similar way. For Proposition \ref{Prop:invariants3D2} it is solved by comparing the pairs  $(\kappa_{s_1},\lambda_{s_1})$ and $(\kappa_{s_2},\lambda_{s_2})$ where $\kappa_s$ and $\lambda_s$ are defined by $J_s=\kappa_{s}(I_s)$ and $K_s=\lambda_s(I_s)$.


\section{Recurrent Weyl structures with additional symmetries}\label{sect:symmetric}

In this section we give a local description of the recurrent Weyl manifolds with additional infinitesimal symmetries. We start with those of dimension $\geq 4$.

As we saw in Proposition \ref{psl2}, the symmetry algebra of any  non-closed recurrent Lorentzian Weyl structure on a manifold of dimension $n+2 \geq 4$ contains the Lie algebra spanned by \eqref{eq:Killing}, which is of dimension $(2n-1) + \binom{n-1}{2}$ and abstractly isomorphic to $\mathfrak{heis}(n-1) \rtimes \mathfrak{so}(n-1)$. The remaining vector fields preserving the coordinate form of the Weyl structures are given in \eqref{eq:aff1psl2}, and span a Lie algebra isomorphic to $\aff \oplus \sl$. We repeat them here for convenience: 
\begin{equation*}
\begin{gathered}
Z_1 = \partial_t, \quad Z_2=2 t \partial_t+3 \sum_{i=1}^{n-1} x^i \partial_{x^i}+6v \partial_v, \\
Z_3=\partial_u, \quad Z_4=2 u \partial_u+\sum_{i=1}^{n-1} x^i \partial_{x^i}, \quad Z_5=u^2 \partial_u+u \sum_{i=1}^{n-1} x^i \partial_{x^i} -\frac{1}{2} \sum_{i=1}^{n-1} (x^i)^2 \partial_v.
\end{gathered}
\end{equation*} 

The general linear combination $\sum_{i=1}^5 a_i Z_i$ is a symmetry of the recurrent Weyl structure defined by the function $\psi$ if and only if $\psi$ satisfies
\begin{equation}
(a_1+2t a_2) \psi'(t)+a_5 \psi(t)^2-2 a_4 \psi(t) +a_3 = 0. \label{eq:psi}
\end{equation}
If either $a_1=a_2=0$ or $a_3=a_4=a_5=0$ it follows that $\psi$ is constant. Since we require $\psi'(t)>0$ it is clear that the Lie algebra $\mathfrak g \subset \aff \oplus \sl$ of additional symmetries satisfies $\mathfrak g \cap \aff =\{0\}$ and $\mathfrak g \cap \sl =\{0\}$. This implies in particular that $\dim \mathfrak g \leq 2$ and that both projections
\[ \mathfrak g \to  \aff, \qquad \mathfrak g \to \sl\]
are injective. We summarize this in a Lemma.

\begin{lem}
The Lie algebra of infinitesimal symmetries of a  non-closed recurrent Lorentzian Weyl structure is of the form $\mathfrak{so}(n-1) \ltimes (\mathfrak{g} \ltimes \mathfrak{heis}(n-1))$ with $\mathfrak g \subset \aff \oplus \sl$, $\mathfrak g \cap \aff=\{0\}$, $\mathfrak g \cap \sl=\{0\}$ and $\dim \mathfrak g \leq 2$. 
\end{lem}

We can use our remaining coordinate freedom of $\SAff \times \PSL \times \mathbb Z_2$ to simplify the search for recurrent Weyl structures with additional symmetries. Let us first consider the case when $\dim \mathfrak g =2$. The transformations of $\PSL$ act transitively on the 2-dimensional Lie subalgebras of $\mathfrak{sl}(2)$, so we may assume that the projection of $\mathfrak g$ to $\mathfrak{sl}(2)$ is spanned by $Z_3$ and $Z_4$ (after a coordinate change preserving the form \eqref{eq:Weyl} and the condition $\psi'(t)>0$).  
Since both $\langle Z_1, Z_2\rangle$ and $\langle Z_3, Z_4 \rangle$ contain a unique 1-dimensional ideal (spanned by $Z_1$ and $Z_3$, respectively), it follows that \[\mathfrak g=\langle Z_1+C_1 Z_3, Z_2 + Z_4+C_2 Z_3 \rangle,\]
where the coefficient of $Z_4$ is determined by the requirement that $\mathfrak{g}$ is closed under the Lie bracket. By further applying transformations from $\SAff \times \PSL$, and changing basis, the Lie algebra can be brought to the form $\langle Z_1\pm Z_3, Z_2+Z_4\rangle$. Only one of these preserves a Weyl structure satisfying  $\psi'(t)>0$, namely $\mathfrak g= \langle Z_1- Z_3, Z_2+Z_4\rangle$. In this case the Weyl structure \eqref{eq:Weyl} is $\mathfrak g$-invariant if and only if $\psi$ satisfies the ODEs 
\[ \psi'(t)-1=0, \qquad 2t \psi'(t) -2 \psi(t) = 0.\]
The unique solution of this pair of ODEs is $\psi(t)=t$. It follows that all  non-closed recurrent Lorentzian Weyl structures with $\dim \mathfrak g = 2$ are locally equivalent to the one defined by $\psi(t)=t$. It is not difficult to see that the full Lie algebra of infinitesimal symmetries acts transitively, implying that the Weyl manifold is locally homogeneous.

\begin{prop}\label{prop:homogeneous}
Let $(M,h,\omega_h)$ be a locally homogeneous  non-closed recurrent Lorentzian Weyl structure of dimension $n+2 \geq 4$. In a neighborhood of each point in $M$, there exists local coordinates so that the pair $(h,\omega_h)$ takes the form 
\[ h=(dt)^2+\frac{1}{(u+t)^2} \left( 2 dv du + \sum_{i=1}^{n-1} (dx^i)^2 \right), \qquad  \omega_h = \frac{dt}{u+t},\]
for $u+t >0$. It is invariant under the Lie algebra spanned by the vector fields
\begin{gather*}
\partial_v, \qquad \partial_{x^i}, \qquad 
x^i \partial_v-u \partial_{x^i}, \qquad x^i \partial_{x^j}-x^j \partial_{x^i}, \\ \partial_t-\partial_u, \qquad t\partial_t+u\partial_u+2 \sum_{i=1}^{n-1} x^i \partial_{x^i}+3 v\partial_v.
\end{gather*}
All other  non-closed  recurrent Lorentzian Weyl structures have symmetry algebra of strictly lower dimension.
\end{prop}
Most of the vector fields in Proposition \ref{prop:homogeneous} are Killing fields of the metric $h$. The only exception to this is the last vector field, which is a homothety. However, all of these vector fields are Killing fields of the conformally related metric $\tilde h=\frac{1}{(u+t)^2} h$. Furthermore, they are the only Killing fields of $\tilde h$, and they annihilate the 1-form $\omega_{\tilde h}$.

%

Consider now the case when $\dim \mathfrak g=1$. By applying $\SAff$ to a general element $a_1 Z_1+a_2 Z_2 \in \aff$, it is clear that we can set either $a_1=0$ or $a_2=0$. Thus we are left with two non-equivalent 1-dimensional Lie subalgebras of  $\aff$:
\[\langle Z_1 \rangle, \qquad \langle Z_2 \rangle.\]  
Classification of elements in $\mathfrak{sl}(2)$ up to transformations of $\PSL$ is equivalent to the classification of binary quadratic forms, which is well-known (see for example \cite{Olver}). From this classification it follows that the 1-dimensional Lie subalgebras of $\mathfrak{sl}(2)$ can be brought to one of the following, nonequivalent, normal forms: 
\[ \langle Z_3 \rangle, \qquad \langle Z_4 \rangle, \qquad \langle Z_3+Z_5 \rangle.\] 
Since we require that both projections $\mathfrak g \to \aff$ and $\mathfrak g \to \sl$ are nontrivial, we are left with 6 cases to investigate:
\begin{description}
\item[Case 1, $\mathfrak g \subset \langle Z_1 \rangle \oplus \langle Z_3 \rangle$] By using a scaling transformation from $\SAff$ we can set $\mathfrak g= \langle Z_1 \pm Z_3 \rangle$. Only one of these preserves recurrent Weyl structures of the form \eqref{eq:Weyl} with $\psi'(t)>0$, namely $\mathfrak g=\langle Z_1-Z_3 \rangle$. The invariant recurrent Weyl structures are given by \[\psi(t)=t+C.\] 
These Weyl structures are locally homogeneous, and equivalent to the one in Proposition \ref{prop:homogeneous}. The integration constant $C$ can be set to 0 by using a translation from $\SAff$ (which preserves the Lie algebra $\mathfrak g$). 

\item[Case 2, $\mathfrak g \subset \langle Z_1 \rangle \oplus \langle Z_4 \rangle$] After applying a scaling transformation in $\SAff$, we have $\mathfrak g = \langle 2Z_1 \pm Z_4 \rangle$. The sign can be fixed  by an application of the transformation $(u,x^i,t,v)\mapsto (-u,x^i,-t,-v)$. This results in the Lie algebra $\mathfrak g = \langle 2Z_1 + Z_4 \rangle$ which preserves the recurrent Weyl structures given by 
\[ \psi(t)=C e^{t}. \]
The condition $\psi'(t)>0$ is equivalent to $C>0$, and the integration constant can be set to $C= 1$ by applying a translation from $\SAff$. 

\item[Case 3, $\mathfrak g \subset \langle Z_1 \rangle \oplus \langle Z_3+Z_5 \rangle$] The Lie algebra can be brought to the form $\langle Z_1\pm (Z_3+Z_5)\rangle$ by a scaling transformation from $\SAff$. Of these two, only $\mathfrak g =\langle Z_1 - (Z_3+Z_5) \rangle$ admits invariant recurrent Weyl structures of the form \eqref{eq:Weyl} with $\psi'(t)>0$. They are given by 
\[ \psi(t) = \tan(t+C). \]
By applying a translation from $\SAff$ we can set $C=0$.

\item[Case 4, $\mathfrak g \subset \langle Z_2 \rangle \oplus \langle Z_3 \rangle$] We have $\mathfrak g = \langle Z_2-2A Z_3 \rangle$, where we can assume that $A>0$, due to the transformation $(u,x^i,t,v)\mapsto (-u,x^i,-t,-v)$. The general solution of \eqref{eq:psi} for $a_2=1,a_3=-2A,a_1=a_4=a_5=0$ is 
\[\psi(t)=A \ln |t|+C.\]
The integration constant $C$ can be set to $0$ by applying a scaling transformation from $\SAff$. The condition $\psi'(t)>0$ is equivalent to $t>0$ when $A>0$. 
\item[Case 5, $\mathfrak g \subset \langle Z_2 \rangle \oplus \langle Z_4 \rangle$] We have $\mathfrak g = \langle Z_2 + A Z_4 \rangle$, where $A$ is a nonzero constant. The general solution of \eqref{eq:psi} for $a_2=1,a_4=A,a_1=a_3=a_5=0$ is
\[ \psi(t)=C |t|^A.\]
The condition $\psi'(t)>0$ is equivalent to $AC t>0$. The integration constant $C$ can be set to $1$ by applying a scaling from $\SAff$, together with the transformation $(u,x^i,t,v)\mapsto (-u,x^i,-t,-v)$ if necessary. When $A=1$ the recurrent Weyl manifold is locally homogeneous, and equivalent to the one in Proposition \ref{prop:homogeneous}.

\item[Case 6,  $\mathfrak g \subset \langle Z_2 \rangle \oplus \langle Z_3+Z_5 \rangle$] We have $ \mathfrak g = \langle Z_2-2A(Z_3+Z_5) \rangle$, where we can assume that $A>0$, due to the transformation $(u,x^i,t,v)\mapsto (-u,x^i,-t,-v)$.  The general invariant recurrent Weyl structure of the form \eqref{eq:Weyl} is given by 
\[\psi(t)=\tan(A (\ln |t|+C)).\]
We have 
\[\psi'(t) = \frac{A}{t} (1+\psi(t)^2)\]
which means that $\psi'(t)>0$ if and only if $t>0$, when $A>0$. By applying a scaling transformation from $\SAff$ we can set $C=0$.
\end{description}
Thus we obtain the following proposition:

\begin{prop}\label{prop:1dim}
Let $(M,h,\omega)$ be a  non-closed recurrent Lorentzian Weyl manifold of dimension $n+2 \geq 4$. If its symmetry Lie algebra is exactly $(2n+\binom{n-1}{2})$-dimensional, then  in a neighborhood of each point in $M$, there exists local coordinates such that the pair $(h,\omega)$ takes the form \eqref{eq:Weyl} where $\psi$ is given by one of the following: 
\begin{align*}
 \psi(t) &= e^{t}, \\ 
\psi(t) &= \tan(t), \\
 \psi(t)&=A \ln t, \quad A > 0,  \\
 \psi(t) &= \tan(A \ln t), \quad A>0, \\
   \psi(t) &= |t|^{A}, \quad  At>0, \; A \neq 1.
\end{align*}
\end{prop}

Together, Proposition \ref{prop:homogeneous} and Proposition \ref{prop:1dim} give a complete local description of  non-closed recurrent Lorentzian Weyl structures having a larger Lie algebra of symmetries than the generic ones.

Next we describe the recurrent Weyl structures on 3-dimensional manifolds with additional symmetry. Consider first the Weyl structures of the form \eqref{eq:3D1}. They all share the symmetry $\partial_v$. We look for recurrent Weyl structures with additional symmetries of the form 
\[ Y= A(x) \partial_x+B(u) \partial_u + C_1 v \partial_v.\] 
The vector field $Y$ is a symmetry of the recurrent Weyl structure defined by the function $F$ if and only if 
\[A(x) \partial_x F+B(u) \partial_u F+(C_1+B'(u))/2-A'(x)=0.\] 
If either $A(x) \equiv 0$ or $B(u) \equiv 0$ in some neighborhood, we get $\partial_{x} \dot F \equiv 0$ in the same neighborhood, which violates our assumption that $\partial_x \dot F$ is non-vanishing.In a neighborhood of a point where both $A$ and $B$ are non-zero, we may simultaneously change the $x$ and $u$ coordinates (using transformations from \eqref{eq:groupaction3D1}), so that $A(x)\equiv 1$ and $B(u) \equiv 1$. Furthermore, we may set $C_1$ equal to either $-2$ or $0$ by scaling $u$ and $x$ and multiplying $Y$ by a constant. The general  recurrent Weyl structure of the form \eqref{eq:3D1} admitting the symmetry $\partial_x+\partial_u$ is given by 
\[ F(x,u)=\tilde F(x-u),\] 
while the recurrent Weyl structure admitting the symmetry $\partial_x+\partial_u-2 v \partial_v$ is given by 
\[ F(x,u)=x+\tilde F(x-u).\] 
In both cases $\tilde F''(x-u)$ is required to be non-vanishing. 

Now, let's look for recurrent Weyl structures admitting a 2-dimensional Lie algebra of additional symmetries: 
\[\langle Y_1=A_1(x)\partial_x+B_1(u)\partial_u+C_1 v \partial_v, Y_2= A_2(x) \partial_x+B_2(u)\partial_u+C_2 v \partial_v \rangle.\] 
Assuming that either $A_1(x)$ or $A_2(x)$ is nonvanishing in some neighborhood, we can change the basis and make a local coordinate change in $x$ to set $A_1(x)\equiv 1$ and $A_2(x) =x$. This basis satisfies $[Y_1,Y_2]=Y_1$. Since $B_1(u)$ can't vanish identically in this coordinate neighborhood (as it would lead to $\partial_x \dot F=0$), we may in a possibly smaller neighborhood change coordinates to set $B_1(u) \equiv 1$. From the commutation relation it follows that $B_2(u)=u+C$, where the constant $C$ can be set to 0 by a translation in $u$. The constant $C_1$ is now required to vanish, due to the commutation relation. The recurrent Weyl structures of the form \eqref{eq:3D1} admitting the symmetries $\partial_x+\partial_u$ and $x \partial_x+u\partial_u+C_2 v \partial_v$
are given by 
\[F(x,u)=\frac{1-C_2}{2} \ln|u-x|+C,\] 
where $C_2 \neq 1$, and the integration constant $C$ can be set equal to $0$ by simultaneously scaling $x$ and $u$ ($C_2 =1$ would imply that $F$ is constant). 
This defines locally homogeneous recurrent Weyl manifolds:
\begin{equation*}
 g = 2 dv du + |u-x|^{C_2-1} (dx)^2, \qquad \omega_g = \frac{1-C_2}{2(u-x)} du.
\end{equation*}
For a given $C_2$, this is defined for $u-x>0$ and $u-x<0$, respectively. The coordinate transformation $(u,x,v) \mapsto (-u,-x,-v)$ lets us identify these two recurrent Weyl structures, so we end up with 
\begin{equation*}
 g = 2 dv du + (u-x)^{C_2-1} (dx)^2, \qquad \omega_g = \frac{1-C_2}{2(u-x)} du, \qquad u-x>0.
\end{equation*}

Lastly, we consider the Weyl structures of the form \eqref{eq:3D2}. Assume that such a Weyl structure has an additional symmetry of the form 
\[Y=A_1\partial_u+A_2 \partial_v+A_3(u\partial_u-v \partial_v)+A_4(u \partial_u+v\partial_v+x\partial_x).\]
The vector field $Y$ is a symmetry of the Weyl structure defined by $a(u)$ and $c(u)$ if and only if the following two equations hold: 
\begin{equation} 
\begin{split} 
(A_1+(A_3+A_4)u) a'(u)+(A_3+2 A_4) a(u)=0, \\ (A_1+(A_3+A_4)u ) c'(u)+(2A_3+A_4) c(u)+A_2 a(u)=0.
\end{split} \label{eq:symmetry3D2}
\end{equation}
If $A_4 \neq 0$, we can multiply $Y$ with a constant factor to set $A_4$ equal to $1$. If $A_4=0$ and $A_3\neq 0$, then we can set $A_3=1$ in the same way.  Next, unless both $A_3=0$ and $A_4=0$, we can independently make translations in $u$ and $v$ to set $A_1=0$ and $A_2=0$. We end up with two distinct (families of) vector fields to consider:
\[u \partial_u+v\partial_v+x\partial_x + A_3(u\partial_u-v\partial_v), \quad u\partial_u-v\partial_v.\]
Alternatively, if both $A_3=0$ and $A_4=0$, then we can multiply $Y$ by a constant, setting $A_1$ equal to $0$ or $1$. Then, if $A_1=0$, we can set $A_2$ equal $1$ in the same way. If $A_1=1$ and $A_2 \neq 0$, we can set $A_2$ equal to $\pm 1$ by applying a transformation in the flow of $2u \partial_u+x\partial_x$. We obtain the following four distinct vector fields:
\[ \partial_v, \qquad \partial_u, \qquad \partial_u-\partial_v, \qquad \partial_u+\partial_v.\]

By solving equation \eqref{eq:symmetry3D2} for each of these cases, we find the symmetric recurrent Weyl structures. If $A_1=A_3=A_4=0, A_2=1$, then \eqref{eq:symmetry3D2} implies $a(u)\equiv 0$, which is inconsistent with our assumption that $a$ is non-vanishing. If $A_2=A_3=A_4=0, A_1=1$, the solutions are given by $a(u) \equiv A, c(u) \equiv C$ with $A,C \in \mathbb R$, $A \neq 0$. By applying coordinate transformations from our Lie group, we can normalize the constants to obtain 
\begin{equation}
a(u)\equiv 1, \qquad c(u) \equiv 0.\label{eq:1d1} 
\end{equation}
 It turns out that this recurrent Weyl structure admits two symmetries:
\[ \partial_u, \qquad u\partial_u-3v\partial_v-x\partial_x.\]  
If $A_1=1=- A_2,A_3=A_4=0$ we obtain the general solution $a(u)\equiv A, c(u)= Au+C$. We can apply coordinate transformations from our Lie group to obtain
\begin{equation}
 a(u) \equiv 1, \qquad c(u)= u.\label{eq:1d2}
\end{equation}
If $A_1=1= A_2,A_3=A_4=0$, similar computations lead to 
\begin{equation}
 a(u) \equiv 1, \qquad c(u)=- u.\label{eq:1d2b}
\end{equation}
Choosing constants $A_1=A_2=A_4=0,A_3=1$ gives the general solution $a(u)=A/|u|,c(u)=C/|u|^2$, defined for either $u<0$ or $u >0$. Let us assume $u>0$ and remove the absolute value sign. Each solution defined for $u<0$ is equivalent to one of the solutions defined for $u>0$ through the transformation $(u,x,v)\mapsto (-u,x,-v)$. We can set $A= 1$ using the flow of $u\partial_u+v\partial_v+x\partial_x$, and the transformation $(u,x,v)\mapsto(u,-x,v)$ if necessary. After this normalization we get the following solutions:
\begin{equation}
a(u)=\frac{1}{u}, \quad c(u)= \frac{C}{u^2}, \quad u>0.\label{eq:1d3}
\end{equation}
Now consider the case $A_1=A_2=0,A_4=1$. The general solution to \eqref{eq:symmetry3D2} for this choice of constants is $a(u)=A|u|^{-\frac{A_3+2}{A_3+1}}, c(u)=C|u|^{-\frac{2A_3+1}{A_3+1}}$, which is defined for either $u<0$ or $u>0$. Choosing $u>0$ lets us remove the absolute value sign. Each solution defined for $u<0$ is equivalent to one of the solutions defined for $u>0$ through the transformation $(u,x,v)\mapsto(-u,x,-v)$.  By applying a transformation from the flow of $u \partial_u+v\partial_v+x\partial_x$, and the transformation $(u,x,v)\mapsto (u,-x,v)$ if necessary, we can set $A=1$:
\begin{equation}
a(u)=u^{-\frac{A_3+2}{A_3+1}}, \quad  c(u)= Cu^{-\frac{2A_3+1}{A_3+1}}, \quad u>0.\label{eq:1d4}
\end{equation}
The case $A_3=-1$ is not covered by the above formulas, but for this case the system \eqref{eq:symmetry3D2} implies $a(u) \equiv 0$, so we disregard it. 

 In most cases there is only one infinitesimal symmetry. The only exception is \eqref{eq:1d1} which is also obtained in the special case of \eqref{eq:1d4} given by $A_3=-2$ and $C=0$. This can be verified by inserting the solutions \eqref{eq:1d2}, \eqref{eq:1d2b}, \eqref{eq:1d3} and \eqref{eq:1d4} into \eqref{eq:symmetry3D2} and solving for the constants. 

To sum up, we have the following propositions when $\dim M=3$.
\begin{prop}
Let $(M,g, \omega_g)$ be a  non-closed recurrent Lorentzian Weyl manifold of dimension 3 admitting a 3-dimensional Lie algebra of infinitesimal symmetries. In a neighborhood of each point in $M$, there exists local coordinates so that the pair $(g, \omega_g)$ takes the form 
\begin{equation*}
 g = 2 dv du + (u-x)^{C-1} (dx)^2, \qquad \omega_g = \frac{1-C}{2(u-x)} du, \qquad C \neq 1,
\end{equation*}
with $u-x>0$. The Lie algebra of infinitesimal symmetries is transitive, and spanned by 
\[ \partial_v, \qquad \partial_x+\partial_u, \qquad x \partial_x+u \partial_u+C v \partial_v.\]
\end{prop}
\begin{prop}
Let $(M,g, \omega_g)$ be a  non-closed recurrent Lorentzian Weyl manifold of dimension 3 admitting a 2-dimensional Lie algebra of infinitesimal symmetries. Then one of the following statements hold.
\begin{itemize}
 \item In a neighborhood of a generic point in $M$ there exists local coordinates so that the pair $(g, \omega_g)$ takes the form 
\begin{equation*}
 g = 2 dv du + e^{C x + F(x-u)} (dx)^2, \qquad \omega_g = F'(x-u) du,
\end{equation*}
for some $C \in \{0,1\}$ and some function $F$ with $F''(x-u)$ non-vanishing. In these coordinates, the Lie algebra of symmetries contains the vector fields spanned by 
\[\partial_v, \qquad \partial_x+\partial_u-2 v \partial_v.\] 
\item In a neighborhood of any point in $M$, there exists local coordinates so that the pair $(g, \omega_g)$ takes the form 
\begin{equation*}
g = 2dvdu + (dx)^2 +x\left(v+ \frac{1}{12} x^3\right) (du)^2,\qquad \omega_g=xdu.
\end{equation*}
In these coordinates, the Lie algebra of symmetries is spanned by 
\[ \partial_u, \qquad u \partial_u-3v\partial_v-x\partial_x.\] 
\end{itemize}
\end{prop}
Notice that in the second case the Lie algebra of symmetries is solvable. In the first case, $C=0$ gives an abelian Lie algebra of symmetries while $C=1$ gives a solvable one.

\begin{prop}
Let $(M,g, \omega_g)$ be a  non-closed  recurrent Lorentzian Weyl manifold of dimension 3 admitting a 1-dimensional Lie algebra of infinitesimal symmetries. Then one of the following statements hold.
\begin{itemize}
 \item In a neighborhood of a generic point in $M$ there exists local coordinates so that the pair $(g, \omega_g)$ takes the form 
\begin{equation*}
 g = 2 dv du + e^{-2F(x,u)} (dx)^2, \qquad \omega_g = \dot F du 
\end{equation*}
with $\partial_x \dot F$  non-vanishing. All metrics of this form share the infinitesimal symmetry $\partial_v$. 
\item In a neighborhood of any point in $M$ there exists local coordinates so that the pair $(g, \omega_g)$ takes the form 
\begin{equation*}
\begin{aligned}
g = 2dvdu + (dx)^2 + H (du)^2,\qquad \omega_g=a(u)xdu, \\
H=  a(u)vx+ \frac{1}{12} a^2(u) x^4 -\frac{1}{3} \dot{a}(u) x^3 + c(u)x,
\end{aligned} 
\end{equation*}
where $a(u)$ and $c(u)$ are given by either \eqref{eq:1d2}, \eqref{eq:1d2b} \eqref{eq:1d3} or \eqref{eq:1d4}. These recurrent Weyl structures have the infinitesimal symmetries given, respectively, by
\[\partial_u-  \partial_v, \quad \partial_u +  \partial_v, \quad u\partial_u - v \partial_v, \quad u \partial_u+ x \partial_x+v\partial_v+A_3(u\partial_u-v\partial_v).  \]
\end{itemize}
\end{prop}

\section{Homogeneous spaces} \label{sect:homogeneous}

Consider the locally homogeneous Weyl structure given by $h$ and $\omega_h$ as in Proposition \ref{prop:homogeneous}, and define $g=e^{2 \ln(u+t)}h$. Consider the new coordinates $$t'= \frac{(t+2u)t}{2},\quad v'= v-\frac{t^3}{3}-\frac{ut^2}{2},$$
and denote them again by $t$ and $v$, respectively.
With respect to the new coordinates the metric $g$ and the 1-form $\omega_g$  are given by
$$g=(dt)^2+2dvdu+\sum_{i=1}^{n-1}(dx^i)^2, \qquad \omega_g = -\frac{du}{\sqrt{2t+u^2}}.$$
The symmetry algebra of the Weyl structure is spanned by the vector fields
\begin{gather*}
\partial_v, \qquad \partial_{x^i}, \qquad 
V_i=x^i \partial_v-u \partial_{x^i}, \qquad x^i \partial_{x^j}-x^j \partial_{x^i}, \quad 1\leq i,j\leq n-1, \\ X=\partial_u+t\partial_v-u\partial_t, \qquad Y=2 t\partial_t+u\partial_u+2 \sum_{i=1}^{n-1} x^i \partial_{x^i}+3 v\partial_v.
\end{gather*} 
We will denote this Lie algebra by $\g$.
Let $\g_0$ be the subalgebra of $\g$ spanned by the vector fields $\partial_v$, $\partial_{x^1},\dots, \partial_{x^{n-1}}$, $X$, $Y$. 
The only non-zero Lie brackets of $\g_0$ are
\begin{equation}\label{Liebrg0} [Y,\partial_v]=-3\partial_v,\quad [Y,\partial_{x^i}]=-2\partial_{x^i},\quad 1\leq i\leq n-1, \quad [Y,X]=-X.\end{equation}

Consider the new metric 
$$b=\frac{4}{(2t+u^2)^2}g.$$
The corresponding 1-form is given by
$$\omega_b=\left(\frac{2u}{2t+u^2}-\frac{1}{\sqrt{2t+u^2}}\right)du + \frac{2}{2t+u^2} dt.$$
It is easy to check that the symmetry algebra $\g$ annihilates both $b$ and $\omega_b$. This will imply the equality \eqref{Aut=}.
Let $U$ be the domain in $\Real^{n+2}$ 
defined by the condition
$$2t+u^2>0.$$
The flows of the vector fields $X$ and $Y$ are given respectively by
$$\varphi_s(t,v,x,u)=\left(-\frac{s^2}{2}-us+t,-\frac{s^3}{6}-
\frac{us^2}{2}+ts+v,x,u+s\right),$$
$$\psi_r(t,v,x,u)=\left(e^{2r}t,e^{3r}v,e^{2r}x,e^{r}u\right).$$
It is clear that the flows are defined for each $s,r\in\Real$ and preserve the space $U$. Consequently the symmetry algebra $\g$ may be integrated to the action of the corresponding connected Lie subgroup $G$ on $U$. Consider the connected Lie subgroup $G_0$ of $G$ generated by the translations in the directions of $v,x^1,\dots,x^{n-1}$ and by the flows $\varphi_s$, $\psi_r$. Let us fix $t$ and $u$ satisfying $2t+u^2>0$ and  set 
\[s=  u \sqrt{\frac{2}{2t+u^2}}, \qquad r= -\ln \sqrt{\frac{2}{2t+u^2} }. \]
Then $$\psi_r\varphi_s(1,0,0,0)=(t,v_0,0,u)$$
for certain $v_0\in\Real$. Obviously this implies that $G_0$ acts transitively on $U$. Moreover, the action is simply transitive. 
By the construction, $G_0$ is simply connected, and its Lie algebra is $\g_0$ with the Lie brackets multiplied by $-1$. This shows that the Lie group $G_0$ is isomorphic to the Lie group defined in the Introduction. 
The explicit isomorphism may be described in the following way.   
For $r,v,u\in\Real$, $r>0$, and $y\in\Real^{n-1}$, denote by $(r,w,y,s)$ the transformation
$$(t,v,x,u)\in U\mapsto \psi_{\ln r}\varphi_s\left(t,v+w,x+y,u \right)\in U.$$

Fixing the point $(1,0,0,0)$, we may identify the Lie group $G_0$ with the space $U$:
$$(r,w,y,s)\in G_0\mapsto \psi_{\ln r}\varphi_s\left(1,w,y,0 \right)
=\left(r^2\left(-\frac{s^2}{2}+1\right),r^3\left(-\frac{s^3}{6}+s+w\right),r^2y,rs\right)\in U.$$
The pull-back under this identification of the metric $b$ and the 1-form $\omega_b$ defines on $G_0$ the left-invariant metric and left-invariant 1-form that we denote by the same symbols. 
Consequently, and the Weyl connection $\nabla$ 
becomes a left-invariant connection on $G_0$.

{\bf Proof of Theorem \ref{ThHomog}.}
Let $(M,c,\nabla)$ be a non-closed recurrent Lorentzian Weyl manifold. Suppose that the manifold $M$ is simply connected and $\operatorname{Aut}(M,c,\nabla)$-homogeneous. Denote by $H$ the Lie group $\operatorname{Aut}^0(M,c,\nabla)$ and let $\h$ denote the corresponding Lie algebra of vector fields on $M$. It is clear that $M$ is $H$-homogeneous.

Consider an open subspace $U\subset M$ on which the Weyl structure may be described as in Proposition \ref{prop:homogeneous}. Due to rigidity of conformal Killing vector fields, the restrictions of the vector fields from $\h$ to $U$ is an isomorphism of $\h$ onto a Lie algebra of vector fields on $U$. These vector fields are local symmetries of the induced Weyl structure on $U$, hence we get an injection $\h\hookrightarrow\g$. 
The Lie group $G$ acts on $G_0$ by transformations of the Weyl structure on $G_0$. Denote by $G_e$ the stabilizer of the identity element $e\in G_0$ and by $\g_e\subset\g$ the corresponding subalgebra.  Since $G$ is connected and $G_0$ is simply connected, $G_e$ is connected.
Fix a point $x\in M$. Let $H_x\subset H$ be the stabilizer of $x$. Again, $H_x$ is connected. The injection $\h\hookrightarrow\g$ defines the Lie group homomorphism
$$\varphi:H\to G.$$
It is clear that the image of $\h$ in $\g$ evaluated at a point of $G_0$ coincides with the tangent space at that point, and this imples that the action of $\varphi(H)$ is transitive on $G_0$. 
It is clear that under the injection $\h\hookrightarrow\g$, the Lie algebra $\h_x$ is mapped to a subalgebra of $\g_e$, and, consequently, 
$$\varphi(H_x)\subset G_e.$$
Define the map $$f:M\to G_0$$ in the following way: take an $y\in M$, then there exists an $h\in H$ such that $y=hx$, and we set
$$f(y)=\varphi(h)e.$$
The map $\varphi$ is well-defined. Indeed, if $hx=h_1x$, then $h^{-1}_1h\in H_x$, and $\varphi(h^{-1}_1h)\in G_e$, i.e., $\varphi(h_1)e=\varphi(h)e$.
Consider the maps $f_1:H\to M$ and $f_2:H\to G_0$,
$$f_1(h)=hx,\qquad f_2(h)=\varphi(h)e.$$
Then,
$$f_2=f \circ f_1.$$
Since $f_1$ and $f_2$ are surjective submersions, $f$ is smooth and it is an  surjective submersion as well. Since both $M$ and $G_0$ are simply connected, $f$ is a diffeomorphism. The diffeomorphism $f$ and the Weyl structure on $M$ define a $\varphi(H)$-invariant Weyl structure $(\bar c,\bar\nabla)$ on $G_0$. 
Moreover, by Proposition \ref{prop:homogeneous}, the local Weyl structure on each open subset $U\subset M$ as above is $\g$-invariant. Consequently, the induced Weyl structure $(\bar c,\bar\nabla)$ on $G_0$ is $G$-invariant. 

\begin{lem} If $(\bar c,\bar\nabla)$ is a $G$-invariant recurrent Weyl structure on $G_0$, then it coincides with $(c,\nabla)$.\end{lem}

{\bf Proof.} Consider on $G_0$ the coordinates $t,v,x^1,\dots,x^{n-1},u$ that we considered above on $U$. It is easy to see that the vector fields $$\partial_v,\partial_{x^i}, V_j,X,Y,\quad 1\leq i,j\leq n-1$$
span an ideal, which is the radical of $\g$. We will denote this ideal by $\fraka$. Next,
$$\frakb=[\fraka,\fraka]=\textrm{span}
\{\partial_v,\partial_{x^1},\dots,\partial_{x^{n-1}},V_1,\dots,V_{n-1},X\},$$
$$[\frakb,\frakb]=\textrm{span}\{\partial_v,\partial_{x^1},\dots,\partial_{x^{n-1}}\},$$
$$[[\frakb,\frakb],\frakb]=\Real\partial_v.$$
Consider the structure $(\bar c,\bar\nabla)$. In a neighborhood of each point form $U$ there exist an open subset $V\subset U$ with coordinates $\bar t,\bar v,\bar x^1,\dots,\bar x^{n-1},\bar u$ and a local metric $\bar b$ from $\bar c$ with the 1-form $\omega_{\bar b}$   having the same structure as $b$ and $\omega_b$ above. 
The symmetry algebra $\bar \g$ coincides with $\g$. Consequently, the canonical ideals of $\g$ coincide with similar ideals of $\bar \g$. We conclude that $\partial_{\bar v}$ is proportional to $\partial_v$, each $\partial_{\bar x^i}$ is a linear combination of $\partial_v,\partial_{x^1},\dots,\partial_{x^{n-1}}$, etc. This easily implies that $\bar b$ and $\omega_{\bar b}$ written with respect to the coordinates   
$t,v,x^1,\dots,x^{n-1},u$ coincide respectively with $c_1b$ and $\omega_b$ for some constant $c_1\in\Real$.
This proves the lemma. \qed

 The lemma implies the proof of the theorem. \qed

\section{Einstein-Weyl structures with recurrent curvature tensor} \label{sect:EinsteinWeyl}
An Einstein-Weyl manifold is a Weyl manifold $(M,c,\nabla)$ satisfying
\[ \mathrm{Ric}^{\mathrm{sym}} = \Lambda g\]
for some $g \in c$, where $\mathrm{Ric}^{\mathrm{sym}}$ is the symmetrized part of the Ricci tensor of $\nabla$ and $\Lambda$ is a smooth function. These structures have received a lot of attention, in particular in three dimensions (see for example \cite{BFKN,DMT}). In this section we apply our previous results to give a complete local description of the connected non-closed Einstein-Weyl structures with recurrent curvature tensor. 

It turns out that there are no non-closed Einstein-Weyl manifolds with recurrent curvature tensor when $\dim M \geq 4$. This can be seen either by computing $\mathrm{Ric}^{\mathrm{sym}}$ for $(M,h,\omega_h)$ from Proposition \ref{psl2} and verifying that the equation $\mathrm{Ric}^{\mathrm{sym}} = \Lambda h$ is never satisfied, or by looking at the weighted parallel spinors and applying results from \cite{ParSpin}. 
To find the three-dimensional Einstein-Weyl manifolds with recurrent curvature tensor, we use the local coordinate descriptions of Theorem \ref{dim3Th}. 

The Weyl structure defined by 
\[ g=2dv du + e^{-2F} (dx)^2, \qquad \omega = \dot F du\] 
with $F=F(x,u)$ has the following symmetrized Ricci tensor: 
\[\mathrm{Ric}^{\mathrm{sym}} = -\partial_x \dot F du dx.\]
Clearly, the Weyl structure is Einstein-Weyl if and only if $\Lambda\equiv 0$ and  $\partial_x \dot F \equiv 0$. Thus, there are no nonclosed Einstein-Weyl structures of this form.  On the other hand, the Weyl structure given by 
\[ g=2 dv du + (dx)^2+H(du)^2, \qquad \omega = H_v du\]
with 
\begin{equation}
 H(v,x,u)=a(u) v x + \frac{1}{12} a^2(u) x^4 - \frac{1}{3} \dot a(u) x^3 + c(u) x \label{eq:H}
\end{equation}
is Einstein-Weyl for any choices of $a$ and $c$. Indeed, the symmetrized Ricci tensor is given by 
\[\mathrm{Ric}^{\mathrm{sym}} = -H_{vv} \left(2 dv du+(dx)^2\right)+ \frac{1}{2}\left(H_v^2-2H_{vu} -H_{xx}-H H_{vv} \right) du^2,\]
and it follows that the structure is Einstein-Weyl if and only if 
\[2H_{vu} + H_{xx}-(H H_v)_v=0.\]
In this case $\Lambda = -H_{vv}$. This is essentially a known result: The above differential equation is equivalent to the dKP equation in \cite{DMT} under the point transformation $(v,x,u,H)\mapsto (v,x,-2u,-H)$. It is easy to verify that the function \eqref{eq:H} is a solution of the above PDE for any pair $a,c$. Moreover, if $H$ is of the form \eqref{eq:H}, then $\Lambda\equiv 0$ and thus the Ricci tensor is skew-symmetric. We arrive at the following proposition.

\begin{prop}\label{prop:EW}
  Let $(M,g,\omega)$ be a  non-closed Lorentzian Einstein-Weyl manifold with recurrent curvature tensor. Then $\dim M = 3$ and around each point of $M$ there exist local coordinates in which $g$ and $\omega$ take the form
  \[ g=2 dv du + (dx)^2+H(du)^2, \qquad \omega = H_v du\]
  where $H$ is given by \eqref{eq:H} with non-vanishing $a(u)$. 
\end{prop}
Since the  non-closed Lorentzian Einstein-Weyl manifolds with recurrent curvature tensor are exactly those of the form \eqref{eq:3D2}, the symmetry analysis from the previous section can be directly applied to the recurrent Einstein-Weyl manifolds. In particular, there is a unique one with maximal (2-dimensional) Lie algebra of symmetries: 
\begin{equation*}
g = 2dvdu + (dx)^2 +x\left(v+ \frac{1}{12} x^3\right) (du)^2,\qquad \omega_g=xdu.
\end{equation*}

\vskip0.5cm

{\bf Acknowledgements.} The authors thank Boris Kruglikov for helpful discussions. A.D. was supported by  grant MUNI/A/1092/2021 of Masaryk University.   A.G. was partially supported by the project GF24-10031K of Czech Science Foundation (GA\v{C}R). E.S. was supported by the Norwegian Financial Mechanism 2014-2021 (project registration number 2019/34/H/ST1/00636) and the Tromsø Research Foundation (project ``Pure Mathematics in Norway'').


\begin{thebibliography}{10}
	
	\bibitem{BFKN}
	S. Berjawi, E. V. Ferapontov, B.S. Kruglikov, V. S. Novikov,
	Second-Order PDEs in 3D with Einstein–Weyl Conformal Structure.
	 Ann. Henri Poincar\'e (2021). doi.org/10.1007/s00023-021-01140-2
	 
	 
	 \bibitem{Moroianu1} F. Belgun, A. Moroianu, Weyl-parallel forms, conformal products and Einstein-Weyl manifolds. Asian J.  Math. 15 (2011), no. 4, 499--520.
	 
	 
	 \bibitem{BHMM} J.-P. Bourguignon, O. Hijazi, J.-L. Milhorat, A. Moroianu, S. Moroianu, A Spinorial Approach to Riemannian
	 and Conformal Geometry. Europ. Math. Soc., 2015, 462 pp.
	

\bibitem{Cal1} D.M.J.\,Calderbank, Two dimensional Einstein-Weyl structures. 	Glasgow Math. J. 43 (2001) 419-424.

\bibitem{Cal2}  D.M.J.\,Calderbank, Integrable background geometries. SIGMA 10 (2014), 034, 51 pages.

\bibitem{Andr} A. Dikarev, On holonomy of Weyl connections in Lorentzian signature. Differential Geometry and its Applications 76 (2021), 101759.

\bibitem{ParSpin} A. Dikarev, A. S. Galaev, Parallel spinors on Lorentzian Weyl spaces.  Monatshefte f\"ur Mathematik 196 (2021), 39--58.


\bibitem{DP} M.\,Dunajski, P.\,Plansangkate, 
The quadric ansatz for the mn-dispersionless KP equation, and supersymmetric Einstein-Weyl spaces. J. Phys. A -- Math. Theor. 55 (2022).

\bibitem{DGS} M. Dunajski, J. Gutowski, W. Sabra,
Einstein-Weyl spaces and near-horizon geometry. Classical Quantum Gravity 34 (2017), no. 4.


\bibitem{DMT} M. Dunajski, L.J. Mason, P. Tod, Einstein-Weyl geometry, the dKP equation and twistor theory. J. Geom. Phys. 37 (2001), 63-93.

\bibitem{G-RBook} E. Garc\'ia-R\'io, P. Gilkey, S. Nikcevi\'c and R. V\'azquez-Lorenzo, Applications of
affine and Weyl geometry. Synthesis Lectures on Mathematics and Statistics, 13.
Morgan \& Claypool Publishers, Williston, VT, 2013.

\bibitem{GZ} J.\,Gregorovic, L.\, Zalabova, 
Notes on symmetric conformal geometries.  Archivum Mathematicum 51 (2015), iss. 5, 287--296.


\bibitem{KT} A.J. Keane, B.O.J. Tupper, Conformal symmetry classes for pp-wave spacetimes.   Class. Quantum Grav. 21 (2004), 2037

\bibitem{Handbook}
B. Kruglikov, V. Lychagin, Geometry of differential equations, in D. Krupka, D. Saunders (eds), Handbook of Global Analysis, Elsevier, 2008, 725--771. 

\bibitem{KL}
B. Kruglikov, V. Lychagin, The global Lie-Tresse theorem,
Selecta Math. 22 (2016), 1357--411. 

\bibitem{L07} Th.~Leistner, {On the classification of Lorentzian holonomy groups.} J. Differential Geom. 76 (2007), no. 3, 423--484.

\bibitem{LSch}	Th.\,Leistner, D.\,Schliebner, Completeness of compact Lorentzian manifolds with abelian holonomy. Mathematische Annalen  364 (2016), 1469--1503. 


\bibitem{MOP} P. Meessen, T. Ort\'in, A. Palomo-Lozano, On supersymmetric Lorentzian Einstein–Weyl spaces. J. Geom. Phys. 62 (2012), no. 2, 301--311.

\bibitem{Olver}
P. Olver, Classical Invariant Theory, Cambridge University Press, Cambridge, 1999.

\bibitem{O}
P. Olver, Equivalence, Invariants, and Symmetry, Cambridge University Press, Cambridge, 1995.

\bibitem{Sen24} J.M.M. Senovilla,	Semi-Riemannian manifolds with linear differential conditions on the curvature.
 		Analysis and Mathematical Physics (2024), 14(3), arc. numb. 63

\bibitem{WalkerRec} A.G. Walker, On Ruse’s spaces of recurrent curvature. Proc. London Math. Soc. 52 (1950), 36--64.


\end{thebibliography}
\end{document}